\documentclass{alggeom}
\usepackage{mathtools, microtype, mathrsfs, xfrac, amssymb, enumerate, xspace, stmaryrd, amsthm}
\usepackage[alphabetic]{amsrefs}
\usepackage{ascii}
\usepackage[latin1]{inputenc}
\usepackage{tikz-cd}
\usepackage[all,cmtip]{xy}
\usepackage{hyperref}
\hypersetup{
	colorlinks=true,
	linkcolor=blue,
	filecolor=magenta,      
	urlcolor=cyan,
}
\usepackage{faktor}
\usepackage{cleveref}
\usepackage{comment}
\usepackage{graphicx}
\usepackage{aliascnt}
\usepackage{float}

\numberwithin{equation}{section}

\numberwithin{subsection}{section}

\allowdisplaybreaks[1]

\newtheorem{theorem}{Theorem}[section]
\newtheorem{proposition}[theorem]{Proposition}
\newtheorem{proposition-definition}[theorem]
{Proposition-Definition}
\newtheorem{corollary}[theorem]{Corollary}
\newtheorem{lemma}[theorem]{Lemma}

\theoremstyle{definition}
\newtheorem{definition}[theorem]{Definition}

\newtheorem{remark}[theorem]{Remark}

\theoremstyle{remark}

\newcommand\cA{\mathcal{A}} \newcommand\cB{\mathcal{B}}
\newcommand\cC{\mathcal{C}} 
\newcommand\cE{\mathcal{E}} 
 \newcommand\cH{\mathcal{H}}
 
 \newcommand\cL{\mathcal{L}}
\newcommand\cM{\mathcal{M}} 
\newcommand\cO{\mathcal{O}} \newcommand\cP{\mathcal{P}}
 
\newcommand\cS{\mathcal{S}} 
\newcommand\cU{\mathcal{U}} 
 \newcommand\cX{\mathcal{X}}
\newcommand\cY{\mathcal{Y}} \newcommand\cZ{\mathcal{Z}}

\renewcommand\AA{\mathbb{A}} 

\newcommand\GG{\mathbb{G}} \newcommand\HH{\mathbb{H}}

 \newcommand\PP{\mathbb{P}}

 \newcommand\VV{\mathbb{V}}
 
 \newcommand\ZZ{\mathbb{Z}}

\newcommand\rma{\mathrm{a}}

\newcommand\rmm{\mathrm{m}}

\newcommand\arr{\ifinner\to\else\longrightarrow\fi}

\newcommand\arrto{\ifinner\mapsto\else\longmapsto\fi}

\renewcommand\H{\operatorname{H}}

\newcommand\into{\hookrightarrow}

\newcommand\im[1]{\operatorname{im}(#1)}

\def\displaytimes_#1{\mathrel{\mathop{\times}\limits_{#1}}}

\def\displayotimes_#1{\mathrel{\mathop{\bigotimes}\limits_{#1}}}

\newcommand\curshom{\mathop{\mathcal{H}\hskip -.8pt om}\nolimits}

\newcommand\spec{\operatorname{Spec}}

\newlength{\ignora}

\renewcommand{\setminus}{\smallsetminus}

\newcommand{\gm}{\GG_{\rmm}}

\newcommand{\GL}{\mathrm{GL}}

\newcommand{\PGL}{\mathrm{PGL}}
\newcommand{\ga}{\GG_{\rma}}

\DeclareFontFamily{U}{mathx}{\hyphenchar\font45}
\DeclareFontShape{U}{mathx}{m}{n}{
	<5> <6> <7> <8> <9> <10>
	<10.95> <12> <14.4> <17.28> <20.74> <24.88>
	mathx10
}{}
\DeclareSymbolFont{mathx}{U}{mathx}{m}{n}
\DeclareFontSubstitution{U}{mathx}{m}{n}
\DeclareMathAccent{\widecheck}{0}{mathx}{"71}
\DeclareMathAccent{\wideparen}{0}{mathx}{"75}

\renewcommand{\epsilon}{\varepsilon}

\newcommand{\Mbar}{\overline{\cM}}

\newcommand{\Mtilde}{\widetilde{\mathcal M}}

\newcommand{\Ctilde}{\widetilde{\mathcal C}}

\newcommand{\Htilde}{\widetilde{\cH}}

\newcommand{\Hbar}{\overline{\cH}}

\newcommand{\ThTilde}{\widetilde{\Theta}}

\newcommand{\Detilde}{\widetilde{\Delta}}

\newcommand{\Debar}{\overline{\Delta}}

\newcommand{\ch}[1][*]{\operatorname{CH}^{#1}}

\newcommand{\mt}{\widetilde{\mathcal M}}

\begin{document}

\title{The (almost) integral Chow ring of $\Mtilde_3^7$}
\author{Michele Pernice}
\email{m.pernice(at)kth.se}
\address{KTH, Room 1642, Lindstedtsv\"agen 25,
114 28, Stockholm }



\begin{abstract}
	This paper is the third in a series of four papers aiming to describe the (almost integral) Chow ring of $\Mbar_3$, the moduli stack of stable curves of genus $3$. In this paper, we compute the Chow ring of $\Mtilde_3^7$ with $\ZZ[1/6]$-coefficients.
\end{abstract}

\maketitle
\section*{Introduction}

The geometry of the moduli spaces of curves has always been the subject of intensive investigations, because of its manifold implications, for instance in the study of families of curves. One of the main aspects of this investigation is the intersection theory of these spaces, which can used to solve either geometric, enumerative or arithmetic problems regarding families of curves. In his groundbreaking paper \cite{Mum}, Mumford introduced the intersection theory with rational coefficients for the moduli spaces of stable curves. After almost two decades, Edidin and Graham introduced in \cite{EdGra} the intersection theory of global quotient stacks (therefore in particular for moduli stacks of stable curves) with integer coefficients.

To date, several computations have been carried out. While the rational Chow ring of $\cM_g$, the moduli space of smooth curves, is known for $2\leq g\leq 9$ (\cite{Mum}, \cite{Fab}, \cite{Iza}, \cite{PenVak}, \cite{CanLar}), the complete description of the rational Chow ring of $\Mbar_g$, the moduli space of stable curves, has been obtained only for genus $2$ by Mumford and for genus $3$ by Faber in \cite{Fab}. As expected, the integral Chow ring is even harder to compute: the only complete description of the integral Chow ring of the moduli stack of stable curves is the case of $\Mbar_2$, obtained by Larson in \cite{Lar} and subsequently with a different strategy by Di Lorenzo and Vistoli in \cite{DiLorVis}. It is also worth mentioning the result of Di Lorenzo, Pernice and Vistoli regarding the integral Chow ring of $\Mbar_{2,1}$, see \cite{DiLorPerVis}.

The aim of this series of four papers is to describe the Chow ring with $\ZZ[1/6]$-coefficients of the moduli stack $\Mbar_3$ of stable genus $3$ curves. This provides a refinement of the result of Faber with a completely indipendent method, which has the potential to give a more algorithmic way to compute these Chow rings. The approach is a generalization of the one used in \cite{DiLorPerVis}: we introduce an Artin stack, which is called the stack of $A_r$-stable curves, where we allow curves with $A_r$-singularities to appear. The idea is to compute the Chow ring of this newly introduced stack in the genus $3$ case and then, using localization sequence, find a description for the Chow ring of $\Mbar_3$. The stack $\Mtilde_{g,n}$ introduced in \cite{DiLorPerVis} is cointained as an open substack inside our stack.

\subsection*{Outline of the paper}

This is the third paper in the series. It focuses on computing the Chow ring of the moduli stack of $A_7$-stable curves of genus $3$.

Specifically, \Cref{sec:1} is dedicated to recall the theory of $A_r$-stable curves as discussed in \cite{Per1} and the theory of hyperelliptic $A_r$-stable curves as in \cite{Per2}. In particular, we state again the main results needed for this paper, such as the existence of the contraction morphism and the description of the hyperelliptic locus. In the last part of the section, we discuss the strategy adopted for the computation. Specifically, we introduce a stratification of $\Mtilde_3^7$ by smoooth substacks and state the gluing lemma (see \Cref{lem:gluing}) used for patching together the Chow rings of the strata to get the Chow ring of the whole stack.  

In \Cref{sec:2}, we use the description of the moduli stack of hyperelliptic $A_r$-stable curves as cyclic covers of twisted curves of genus $0$ to describe the Chow ring of the hyperelliptic stratum.
We also compute the normal bundle relative to this stratum which is an essential element for the patching lemma.

In \Cref{sec:3}, we prove that non-hyperelliptic $A_r$-stable curves of genus $3$ without separating nodes are canonically embedded inside $\PP^2$. This gives us a way to describe the open stratum of the stratification as an open inside the quotient stack of quartics in $\PP^2$ by an action of $\GL_3$. Inspired by the computations in \cite{DiLorFulVis}, we manage to compute the Chow ring of the open stratum.

In \Cref{sec:4}, \Cref{sec:5} and \Cref{sec:6}, we describe the remaining strata and compute their Chow rings using such descriptions. We also compute the normal bundle relative to every strata and some cycles that will be essential for the final computation.

Finally, \Cref{sec:7} explains how to apply the gluing lemma to get a description of the Chow ring of the whole stack $\Mtilde_3^7$ using generators and relations coming from the Chow rings of the strata. Such description is the content of \Cref{theo:main} whereas the explicit relations are listed in \Cref{rem:relations-Mtilde}.
 
\section{Preliminaries and strategy}\label{sec:1}

In this section, we recall the definition of the moduli stack $\Mtilde_{g,n}^r$ parametrizing $n$-pointed $A_r$-stable curves of genus $g$ and some results regarding this moduli stack. For a more detailed treatment of the subject, see \cite{Per1}. Furthermore, we recall the definition of the moduli stack $\Htilde_g^r$ parametrizing hyperelliptic $A_r$-stable curves of genus $g$ and the alternative description using the theory of cyclic covers. For a more detailed treatment of the subject, see \cite{Per2}. Finally, we explain the strategy we use for the computation. 

\subsection*{Moduli stack of $A_r$-stable curves}
Fix a nonnegative integer $r$. Let $g$ be an integer with $g\geq 2$ and $n$ be a nonnegative integer.

\begin{definition}	Let $k$ be an algebraically closed field and $C/k$ be a proper reduced connected one-dimensional scheme over $k$. We say the $C$ is a \emph{$A_r$-prestable curve} if it has at most $A_r$-singularity, i.e. for every $p\in C(k)$, we have an isomorphism
		$$ \widehat{\cO}_{C,p} \simeq k[[x,y]]/(y^2-x^{h+1}) $$ 
	with $ 0\leq h\leq r$. Furthermore, we say that $C$ is $A_r$-stable if it is $A_r$-prestable and the dualizing sheaf $\omega_C$ is ample. A $n$-pointed $A_r$-stable curve over $k$ is $A_r$-prestable curve together with $n$ smooth distinct closed points $p_1,\dots,p_n$ such that $\omega_C(p_1+\dots+p_n)$ is ample.
\end{definition}
\begin{remark}
	Notice that a $A_r$-prestable curve is l.c.i by definition, therefore the dualizing complex is in fact a line bundle. 
\end{remark}

We fix a base field $\kappa$ where all the primes smaller than $r+1$ are invertible. Every time we talk about genus, we intend arithmetic genus, unless specified otherwise. We recall a useful fact.

\begin{remark}\label{rem:genus-count}
	Let $C$ be a connected, reduced, one-dimensional, proper scheme over an algebraically closed field. Let $p$ be a rational point which is a singularity of $A_r$-type. We denote by $b:\widetilde{C}\arr C$ the partial normalization at the point $p$ and by $J_b$ the conductor ideal of $b$. Then a straightforward computation shows that \begin{enumerate}
		\item if $r=2h$, then $g(C)=g(\widetilde{C})+h$;
		\item if $r=2h+1$ and $\widetilde{C}$ is connected, then $g(C)=g(\widetilde{C})+h+1$,
		\item if $r=2h+1$ and $\widetilde{C}$ is not connected, then $g(C)=g(\widetilde{C})+h$.
	\end{enumerate}
	If $\widetilde{C}$ is not connected, we say that $p$ is a separating point. Furthermore, Noether formula gives us that $b^*\omega_C \simeq \omega_{\widetilde{C}}(J_b^{\vee})$.
\end{remark}

 We can define $\Mtilde_{g,n}^r$ as the fibered category over $\kappa$-schemes  whose objects are the datum of $A_r$-stable curves over $S$ with $n$ distinct sections $p_1,\dots,p_n$ such that every geometric fiber over $S$ is a  $n$-pointed $A_r$-stable curve. These families are called \emph{$n$-pointed $A_r$-stable curves} over $S$. Morphisms are just morphisms of $n$-pointed curves.

We recall the following description of $\Mtilde_{g,n}^r$. See Theorem 2.2 of \cite{Per1} for the proof of the result. 

\begin{theorem}\label{theo:descr-quot}
	$\Mtilde_{g,n}^r$ is a smooth connected algebraic stack of finite type over $\kappa$. Furthermore, it is a quotient stack: that is, there exists a smooth quasi-projective scheme X with an action of $\GL_N$ for some positive $N$, such that 
	$ \Mtilde_{g,n}^r \simeq [X/\GL_N]$.
\end{theorem}

\begin{remark}\label{rem: max-sing}
Recall that we have an open embedding $\Mtilde_{g,n}^r  \subset \Mtilde_{g,n}^s$ for every $r\leq s$. Notice that $\Mtilde_{g,n}^r=\Mtilde_{g,n}^{2g+1}$ for every $r\geq 2g+1$. 
\end{remark}

The usual definition of the Hodge bundle extends to our setting. See Proposition 2.4 of \cite{Per1} for the proof. As a consequence we obtain a locally free sheaf $\HH_{g}$ of rank~$g$ on $\mt_{g, n}^r$, which is called \emph{Hodge bundle}. 

Furthermore, we recall the existence of the (minimal) contraction morphism. This is Theorem 2.5 of \cite{Per1}.

\begin{theorem}\label{theo:contrac}
	We have a morphism of algebraic stacks 
	$$ \gamma:\Mtilde_{g,n+1}^r \longrightarrow \Ctilde_{g,n}^r$$
	where $\Ctilde_{g,n}^r$ is the universal curve of $\Mtilde_{g,n}^r$. Furthermore, it is an open immersion and its image is the open locus $\Ctilde_{g,n}^{r,\leq 2}$ in $\Ctilde_{g,n}^r$ parametrizing $n$-pointed $A_r$-stable curves $(C,p_1,\dots,p_n)$ of genus $g$ and a (non-necessarily smooth) section $q$ such that $q$ is an $A_h$-singularity for $h\leq 2$.
\end{theorem}

\subsection*{The hyperelliptic locus}

We recall the definition of hyperelliptic $A_r$-stable curves.

\begin{definition}
	Let $C$ be an $A_r$-stable curve of genus $g$ over an algebraically closed field. We say that $C$ is hyperelliptic if there exists an involution $\sigma$ of $C$ such that the fixed locus of $\sigma$ is finite and the geometric categorical quotient, which is denoted by $Z$, is a reduced connected nodal curve of genus $0$. We  call the pair $(C,\sigma)$ a \emph{hyperelliptic $A_r$-stable curve} and such $\sigma$  is called a \emph{hyperelliptic involution}.
\end{definition}

We define $\Htilde_g^r$ as the following fibered category: its objects are the data of a pair $(C/S,\sigma)$ where $C/S$ is an $A_r$-stable curve over $S$ and $\sigma$ is an involution of $C$ over $S$ such that $(C_s,\sigma_s)$ is a $A_r$-stable hyperelliptic curve of genus $g$ for every geometric point $s \in S$. These are called \emph{hyperelliptic $A_r$-stable curves over $S$}. A morphism is a morphism of $A_r$-stable curves that commutes with the involutions. 

Now we introduce another description of $\Htilde_g^r$, useful for understanding the link with the smooth case, using cyclic covers of twisted curves. We refer to \cite{AbOlVis} for the theory of twisted nodal curves, although we consider only twisted curves with $\mu_2$ as stabilizers and with no markings.

\begin{definition}\label{def:hyp-A_r}
	Let $\cZ$ be a twisted nodal curve of genus $0$ over an algebraically closed field. We denote by $n_{\Gamma}$ the number of stacky points of $\Gamma$ and by $m_{\Gamma}$ the number of intersections of $\Gamma$ with the rest of the curve for every $\Gamma$ irreducible component of $\cZ$. Let $\cL$ be a line bundle on $\cZ$ and $i:\cL^{\otimes 2} \rightarrow \cO_{\cZ}$ be a morphism of $\cO_{\cZ}$-modules.  We denote by $g_{\Gamma}$ the quantity $n_{\Gamma}/2-1-\deg\cL\vert_{\Gamma}$. 
	\begin{itemize}
	\item[(a)] We say that $(\cL,i)$ is hyperelliptic if the following are true:
			\begin{itemize}
				\item[(a1)] the morphism $\cZ \rightarrow B\GG_m$ induced by $\cL$ is representable,
				\item[(a2)] $i^{\vee}$ does not vanish restricted to any stacky point.
			\end{itemize}
	\item[(b)] We say that $(\cL,i)$ is $A_r$-prestable and hyperelliptic of genus $g$ if $(\cL,i)$ is hyperelliptic, $\chi(\cL)=-g$ and the following are true:
		\begin{itemize}
			\item[(b1)] $i^{\vee}$ does not vanish restricted to any irreducible component of $\cZ$ or equivalently the morphism $i:\cL^{\otimes 2 }\rightarrow \cO_{\cZ}$ is injective,
			\item[(b2)] if $p$ is a non-stacky node and $i^{\vee}$ vanishes at $p$, then $r\geq 3$ and the vanishing locus $\VV(i^{\vee})_p$ of $i^{\vee}$ localized at $p$ is a Cartier divisor of length $2$;
			\item[(b3)] if $p$ is a smooth point and $i^{\vee}$ vanishes at $p$, then the vanishing locus $\VV(i^{\vee})_p$ of $i^{\vee}$ localized at $p$ has length at most $r+1$.
		\end{itemize}	
	\item[(c)] We say that $(\cL,i)$ is $A_r$-stable and hyperelliptic of genus $g$ if it is $A_r$-prestable and hyperelliptic of genus $g$ and the following are true for every irreducible component $\Gamma$ in $\cZ$:
			\begin{itemize}
				\item[(c1)] if $g_{\Gamma}=0$ then we have $2m_{\Gamma}-n_{\Gamma}\geq 3$,
				\item[(c2)] if $g_{\Gamma}=-1$ then we have
				$m_{\Gamma}\geq 3$ ($n_{\Gamma}=0$).
			\end{itemize}
\end{itemize}
\end{definition}

Let us define now the stack classifying these data. We denote by $\cC(2,g,r)$ the fibered category defined in the following way: the objects are triplet $(\cZ\rightarrow S,\cL,i)$ where $\cZ \rightarrow S$ is a family of twisted curves of genus $0$, $\cL$ is a line bundle of $\cZ$ and $i:\cL^{\otimes 2}\rightarrow \cO_{\cZ}$ is a morphism of $\cO_{\cZ}$-modules such that the restrictions $(\cL_s,i_s)$ to the geometric fibers over $S$ are $A_r$-stable and hyperelliptic of genus $g$. Morphisms are defined as in \cite{ArVis}. 

We recall the following result, which gives us an alternative description of $\Htilde_g^r$. See Proposition 2.14 and Proposition 2.21 of \cite{Per2} for the proof of the result.
\begin{proposition}\label{prop:descr-hyper}
	The fibered category $\cC(2,g,r)$ is isomorphic to $\Htilde_g^r$.
\end{proposition}

Finally, we recall the following theorem which is a conseguence of Proposition 2.23 and Section 3 of \cite{Per2}.

\begin{proposition}\label{prop:smooth-hyp}
	The moduli stack $\Htilde_g^r$ of $A_r$-stable hyperelliptic curves of genus $g$ is smooth and the open $\cH_g$ parametrizing smooth hyperelliptic curves is dense in $\Htilde_g^r$. In particular $\Htilde_g^r$ is connected. Moreover, the natural morphism 
	$$ \Htilde_g^r \longrightarrow \Mtilde_g^r$$
	defined by the association $(C,\sigma) \mapsto C$ is a closed embedding between smooth algebraic stacks.
\end{proposition}

\subsection*{Strategy of the computation}

In the rest of this section we explain the stratey used for the computation of the Chow ring of $\Mtilde_3^7$. The idea is to use a gluing lemma, whose proof is an exercise in homological algebra.

Let $i:\cZ\hookrightarrow\cX$ be a closed immersion of smooth global quotient stacks over $\kappa$ of codimension $d$ and let $\cU:=\cX\setminus \cZ$ be the open complement and $j:\cU \hookrightarrow \cX$ be the open immersion. It is straightforward to see that the pullback morphism $i^*:\ch(\cX)\rightarrow \ch(\cZ)$ induces a morphism $ \ch(\cU) \rightarrow \ch(\cZ)/(c_d(N_{\cZ|\cX}))$, where $N_{\cZ|\cX}$ is the normal bundle of the closed immersion. This morphism is denoted by $i^*$ by abuse of notation. 

Therefore, we have the following commutative diagram of rings:
$$
\begin{tikzcd}
\ch(\cX) \arrow[d, "j^*"] \arrow[rr, "i^*"] &  & \ch(\cZ) \arrow[d, "q"]     \\
\ch(\cU) \arrow[rr, "i^*"]                  &  & \frac{\ch(\cZ)}{(c_d(N_{\cZ|\cX}))}
\end{tikzcd}
$$
where $q$ is just the quotient morphism.

\begin{lemma}\label{lem:gluing}
  In the situation above, the induced map 
  $$\zeta: \ch(\cX)\longrightarrow \ch(\cZ)\times_\frac{\ch(\cZ)}{(c_d(N_{\cZ|\cX}))} \ch(\cU)$$
  is surjective and $\ker \zeta= i_* {\rm Ann}(c_d(N_{\cZ|\cX}))$. In particular, if $c_d(N_{\cZ|\cX})$ is a non-zero divisor in $\ch(\cZ)$, then $\zeta$ is an isomorphism. 
 \end{lemma}

From now on, we refer to the condition \emph{$c_d(N_{\cZ|\cX})$ is not a zero divisor} as the gluing condition.

\begin{remark}
	Notice that if $\cZ$ is a Deligne-Mumford separated stack, the gluing condition is never satisfied because the rational Chow ring of $\cZ$ is isomorphic to the one of its moduli space (see \cite{Vis1} and \cite{EdGra}).
	
	However, there is hope that the positive dimensional stabilizers  we have in $\Mtilde_g^r$ allow the gluing condition to occur. For instance, consider $\Mtilde_{1,1}^2$, or the moduli stack of genus $1$ marked curves with at most $A_2$-singularities. We know that 
	$$ \Mtilde_{1,1}^2 \simeq [\AA^2/\gm]$$ 
	therefore its Chow ring is a polynomial ring, which has plenty of non-zero divisors. See \cite{DiLorPerVis} for the description of $\Mtilde_{1,1}$ as a quotient stack.
\end{remark}

This is the reason why we introduced $\Mtilde_g^r$, which is a non-separated stacks. We are going to compute the Chow ring of $\Mtilde_3^r$ for $r=7$ using a stratification for which we can apply \Cref{lem:gluing} iteratively.

For the rest of the paper, we denote by $\Mtilde_{g,n}$ the stack $\Mtilde_{g,n}^{2g+1}$ which is the largest moduli stack of $n$-pointed $A_r$-stable curves of genus $g$ we can consider (see \Cref{rem: max-sing}). The same notation will be used for $\Htilde_g^r$. We denote by $\Ctilde_{g,n}$ the universal curve of $\Mtilde_{g,n}$. Therefore $\Mtilde_3^7$ will be denote $\Mtilde_3$. Recall that our base field $\kappa$ has characteristic different from $2,3,5,7$.

Every Chow ring is considered with $\ZZ[1/6]$-coefficients unless otherwise stated. This assumption is not necessary for some of the statements, but it makes our computations easier. We do not know if the gluing condition still holds with integer coefficients. 

We want to costruct a stratification of $\Mtilde_3$ to which we can apply \Cref{lem:gluing} to obtain the description of the whole Chow ring.

First of all, we recall the definitions of some substacks of $\Mbar_3$:
\begin{itemize}
	\item $\overline{\Delta}_1$ is the codimension 1 closed substack of $\Mbar_3$ classifying stable curves with a separating node, 
	\item $\overline{\Delta}_{1,1}$ is the codimension 2 closed substack of $\Mbar_3$ classifying stable curves with two separating nodes, 
	\item $\overline{\Delta}_{1,1,1}$ is the codimension 3 closed substack of $\Mbar_3$ classifying stable curves with three separating nodes. 
\end{itemize}

An easy combinatorial computation shows that we cannot have more than $3$ separating nodes for genus $3$ stable curves. The same substacks can be defined for $\Mtilde_3^r$ for any $r$ (in particular for $r=7$), but we need to prove that they are still closed inside $\Mtilde_3^r$. This is a conseguence of Lemma 4.7 of \cite{Per1}.

We denote by $\Detilde_{1,1,1}\subset \Detilde_{1,1} \subset \Detilde_1$ the natural generalization of $\Debar_{1,1,1} \subset \Debar_{1,1}\subset \Debar_1$ in $\Mtilde_3$.

We also consider $\Htilde_3^7$ as a stratum of the stratification. The following diagram
$$
\begin{tikzcd}
	&                            & \Htilde_3^7 \arrow[rd] &             \\
	{\Detilde_{1,1,1}} \arrow[r] & {\Detilde_{1,1}} \arrow[r] & \Detilde_1 \arrow[r]   & \Mtilde_3^7
\end{tikzcd}
$$
represents the poset associated to the stratification. As before, we write $\Htilde_3$ instead of $\Htilde_3^7$.

Now, we describe the strategy used to compute the Chow ring of $\Mtilde_3$. Our approach focuses firstly on the computation of the Chow ring of $\Mtilde_3 \setminus \Detilde_1$. This is the most difficult part as far as computations is concerned. We first compute the Chow ring of $\Htilde_3 \setminus \Detilde_1$, which can be done without the gluing lemma. Then we apply the gluing lemma to $\Mtilde_3 \setminus (\Htilde_3 \cup \Detilde_1)$ and $\Htilde_3 \setminus \Detilde_1$ to get a description for the Chow ring of $\Mtilde_3 \setminus \Detilde_1$.

Notice that neither $\Detilde_1$ and $\Detilde_{1,1}$ are smooth stacks therefore we cannot use \Cref{lem:gluing} to compute their Chow rings. Nevertheless, both  $\Detilde_1 \setminus \Detilde_{1,1}$ and $\Detilde_{1,1} \setminus \Detilde_{1,1,1}$ are smooth, therefore we apply \Cref{lem:gluing} to $\Detilde_1 \setminus \Detilde_{1,1}$ and $\Mtilde_3\setminus \Detilde_1$ to describe the Chow ring of $\Mtilde_3 \setminus \Detilde_{1,1}$, and then apply it again to $\Mtilde_3 \setminus \Detilde_{1,1}$ and $\Detilde_{1,1} \setminus\Detilde_{1,1,1}$. Finally, the same procedure allows us to glue also $\Detilde_{1,1,1}$ and get the description of the Chow ring of $\Mtilde_3$.
 
In the rest of the paper, we describe the following strata and their Chow rings:
\begin{itemize}
	\item $\Htilde_3\setminus \Detilde_1$,
	\item $\Mtilde_3\setminus (\Htilde_3 \cup \Detilde_1)$,
	\item $\Detilde_1\setminus \Detilde_{1,1}$,
	\item $\Detilde_{1,1}\setminus \Detilde_{1,1,1}$,
	\item $\Detilde_{1,1,1}$,
\end{itemize}
and in the last section we give the complete description of the Chow ring of $\Mtilde_3$ explaining the computations involved in the gluing procedure.

\section{Chow ring of $\Htilde_3\setminus\Detilde_1$}\label{sec:2}
 
In this section we are going to describe $\Htilde_3 \setminus \Detilde_1$ and compute its Chow ring.

Recall that we have an isomorphism (see \Cref{prop:descr-hyper}) between $\Htilde_3^7$ and $\Ctilde_3^7$. We are going to describe $\Htilde_3\setminus \Detilde_1$ as a subcategory of $\cC(2,3,7)$ through the isomorphism cited above.

Consider the natural morphism (see the proof of \Cref{prop:smooth-hyp})
$$ \pi_3: \Htilde_3 \setminus \Detilde_1 \rightarrow \cP_3$$ 
where $\cP_3$ is the moduli stack parametrizing pairs $(Z/S,\cL)$ where $Z$ is a twisted curve of genus $0$ and $\cL$ is a line bundle on $Z$. The idea is to find the image of this morphism. First of all, we can restrict to the open of $\cP_3$ parametrizing pairs $(Z/S,\cL)$ such that $Z/S$ is an algebraic space, because we are removing $\Detilde_1$. In fact, if there are no separating points, $Z$ coincides with the geometric quotient of the involution (see the proof of Proposition 2.14 of \cite{Per2}). Moreover, we prove an upper bound to the number of irreducible components of $Z$.

\begin{lemma}\label{lem:max-comp}
	Let $(C,\sigma)$ be a hyperelliptic $A_r$-stable curve of genus $g$ over an algebraically closed field and let $Z:=C/\sigma$ be the geometric quotient. If we denote by $v$ the number of irreducible components of $Z$, then we have $v\leq 2g-2$.  Furthermore, if $C$ has no separating nodes, we have that $v\leq g-1$.
\end{lemma}

\begin{proof}
	Let $\Gamma$ be an irreducible component of $Z$. We denote by $g_{\Gamma}$ the genus of the preimage $C_{\Gamma}$ of $\Gamma$ through the quotient morphism, by $e_{\Gamma}$ the number of nodes lying on $\Gamma$ and by $s_{\Gamma}$ the number of nodes on $\Gamma$ such that the preimage is either two nodes or a tacnode (see Proposition 3.2 of \cite{Per2}). We claim that the stability condition on $C$ implies that
	$$ 2g_{\Gamma}-2+e_{\Gamma}+s_{\Gamma}>0$$
	for every $\Gamma$ irreducible component of $Z$. If $C_{\Gamma}$ is integral, then it is clear. Otherwise, $C_{\Gamma}$ is a union of two projective line meeting on a subscheme of length $n$. In this situation, $s_{\Gamma}=e_{\Gamma}:=m$ and the inequality is equivalent to $n+m\geq 3$, which is the stability condition for the two projective lines. The identity  $$g=\sum_{\Gamma}(g_{\Gamma}+\frac{s_{\Gamma}}{2})$$
     and the claim imply the thesis.

	 Suppose $(C,\sigma)$ is not in $\Detilde_1$. 
	 Notice that the involution $\sigma$ commutes with the canonical morphism (which is globally defined thanks to Proposition 3.36 of \cite{Per2}), because it does over the open dense substack of $\Htilde_g$ parametrizing smooth curves. If we consider the factorization of the canonical morphism  
	 $$ C \rightarrow Z \rightarrow \PP(\H^0(C,\omega_C))=\PP^{g-1}$$ 
	 we know that $\cO_Z(1)$ has degree $g-1$ because the quotient morphism $C\rightarrow Z$ has degree $2$ and $Z\rightarrow \PP^{g-1}$ is finite. It follows that $v\leq g-1$. 
\end{proof}

\begin{remark}
	The first inequality is sharp even for ($A_1$-)stable hyperelliptic curves. Let $v:=2m$ be an even number and let $(E_1,e_1),(E_2,e_2)$ be two smooth elliptic curves and $P_1,\dots,P_{2m-2}$ be $2m-2$ projective lines. We glue $P_{2i-1}$ to $P_{2i}$ in $0$ and $\infty$ for every $i=1,\dots,m-1$ and we glue $P_{2i}$ to $P_{2i+1}$ in $1$ for every $i=1,\dots,m-2$. Finally we glue $(E_1,e_1)$ to $(P_1,1)$ and $(E_2,e_2)$ to $(P_{2m-2},1)$. It is clear that the curve is $A_1$-stable hyperelliptic of genus $m+1$ and its geometric quotient has $2m$ components. The odd case can be dealt similarly.
	\begin{figure}[H]
		\centering
		\includegraphics[width=1\textwidth]{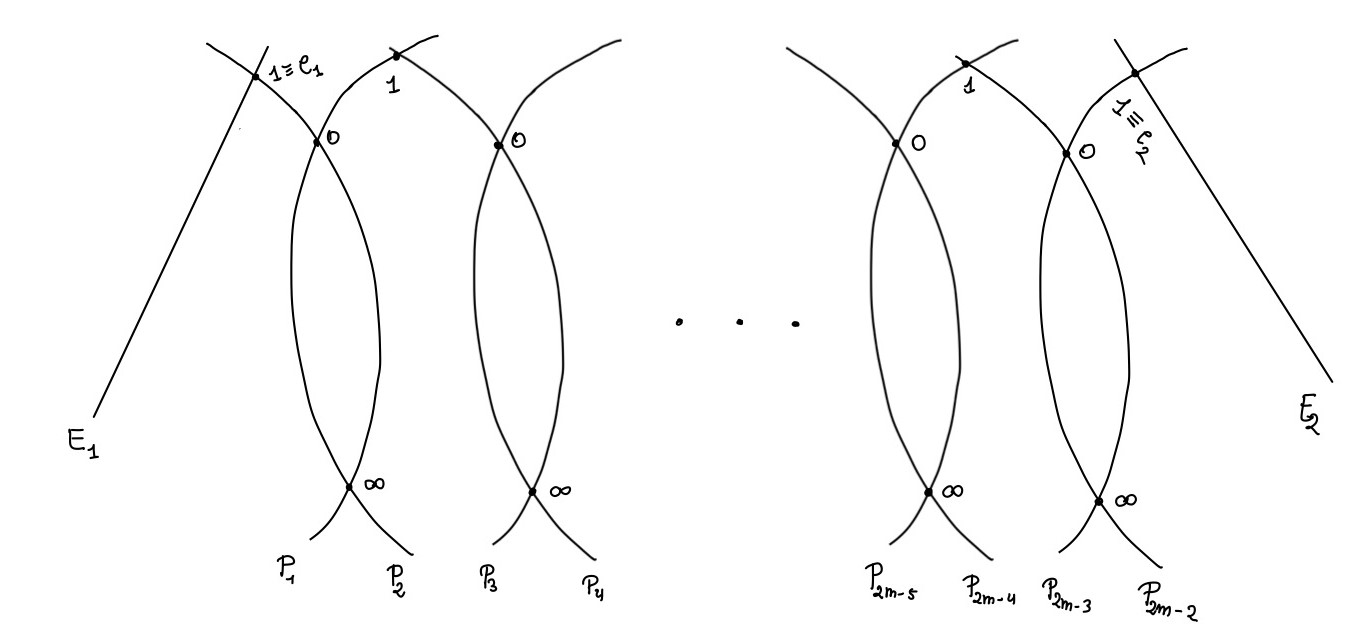}
		\label{fig:Contro}
	\end{figure}
	The same is true for the second inequality. Suppose $v:=2m$ is an even positive integer. Let $(E_1,e_1,f_1),(E_2,e_2,f_2)$ be two $2$-pointed smooth genus $1$ curves and $P_1,\dots,P_{2m-2}$ be $2m-2$ projective lines. Now we glue $P_{2i-1}$ to $P_{2i}$ in $0$ and $\infty$ for every $i=1,\dots,m-1$ and we glue $P_{2i}$ to $P_{2i+1}$ in $1$ and $-1$ for every $i=1,\dots,v-2$. Finally we glue $(E_1,e_1,f_1)$ to $(P_1,1,-1)$ and $(E_2,e_2,f_2)$ to $(P_{2m-2},1,-1)$. It is clear that the curve is $A_1$-stable hyperelliptic of genus $2m+1$ and its geometric quotient has $2m$ components. The odd case can be dealt similarly.
	\begin{figure}[H]
		\centering
		\includegraphics[width=1\textwidth]{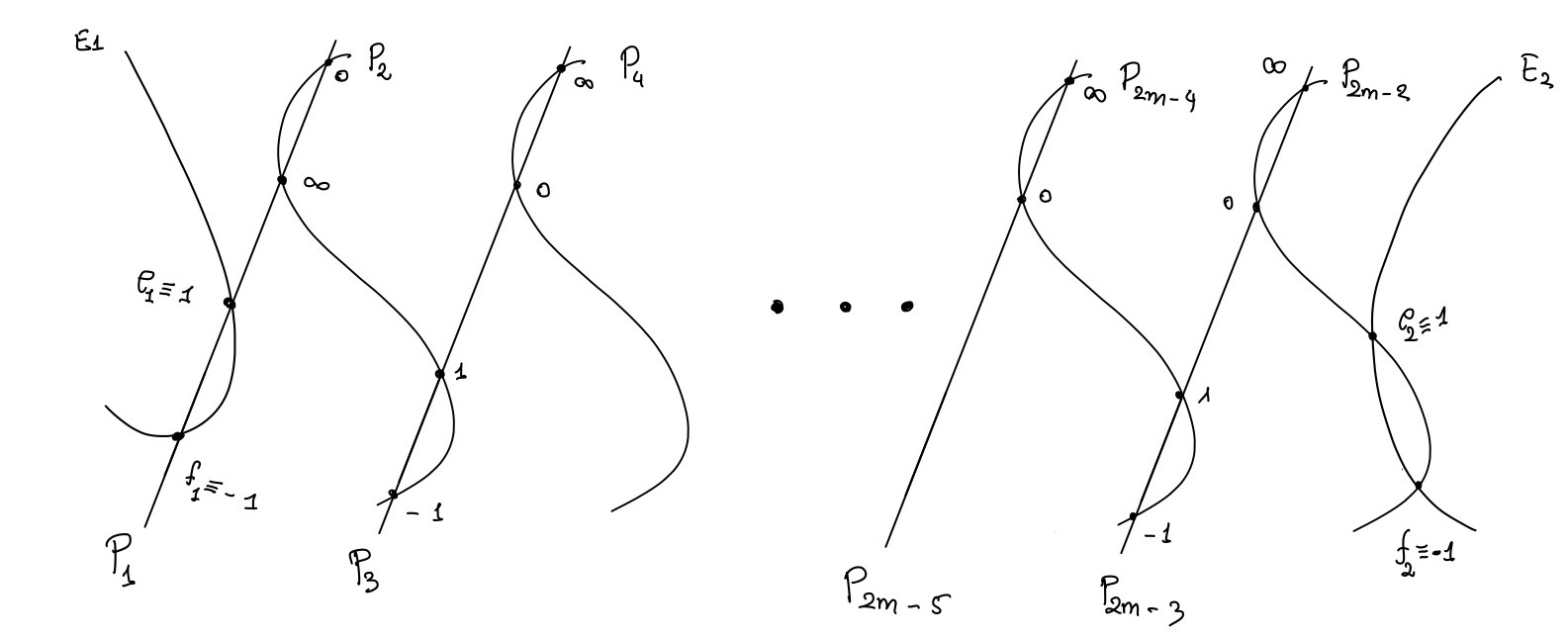}
		\label{fig:Contro1}
	\end{figure}
\end{remark}

\Cref{lem:max-comp} assures us that the geometric quotient $Z$ has at most two irreducible components if $C$ has genus $3$ and does not have separating nodes. Therefore, the datum $(Z/S,\cL,i)$ is in the image of $\pi_3$ only if the fiber $Z_s$ has at most $2$ irreducible components for every geometric point $s \in S$. Moreover, we know that not all the pairs $(\cL,i)$ give us a $A_r$-stable hyperelliptic curve of genus $3$. We need to translate the conditions in \Cref{def:hyp-A_r} in our setting. Because there are no stacky points in $Z$, the conditions in (a) are empty. Furthermore, if $Z$ is integral, we have that $\cL$ is non-canonically isomorphic to $\cO_{\PP^1}(-4)$ because $\chi(L)=-3$ and the only condition on the section $i$ is that is different from $0$ because $r=7$.

Suppose $Z$ is the union of two irreducible components $\Gamma_1$ and $\Gamma_2$  intersecting in a single node. Let us call $n_1$ and $n_2$ the degrees of the restrictions of $\cL$ to $\Gamma_1$ and $\Gamma_2$ respectively. Thus (b1)+(c1) implies that $n_1\leq -2$ and $n_2\leq -2$. Finally the condition $\chi(\cL)=-3$ gives us $n_1=n_2=-2$. 

We denote by $\cM_0^{\leq 1}$ the moduli stack parametrizing  families $Z/S$ of genus $0$ nodal curves with at most $1$ node and $\cM_0^{1}$ the closed substack of $\cM_0^{\leq 1}$ parametrizing families $Z/S$ with exactly $1$ node. We follows the notation as in \cite{EdFul2}. 

Finally, let us denote by $\cP_3'$ the moduli stack parametrizing pairs $(Z/S,\cL)$ where $Z/S$ is a family of genus $0$ nodal curves with at most 1 node and $\cL$ is a line bundle that restricted to the geometric fibers over every point $s \in S$ has degree $-4$ if $Z_s$ is integral or bidegree $(-2,-2)$ if $Z_s$ is reducible.

Therefore $\Htilde_3\setminus\Detilde_1$ can be seen as an open substack of the vector bundle $\pi_*\cL^{\otimes -2}$ of $\cP_3'$, where $(\pi:\cZ \rightarrow \cP_3',\cL)$ is the universal object of $\cP_3'$. We can explicitly describe the complement of $\Htilde_3 \setminus \Detilde_1$ inside $\VV(\pi_{*}\cL^{\otimes-2})$. The only conditions the section $i$ has to verify in this particular case are (b1) and (b2) of \Cref{def:hyp-A_r}. The closed substack of $\VV(\pi_*\cL^{\otimes -2})$ that parametrizes sections which do not verify (b1) or (b2) can be described  as the union of two components. The first one is the zero section of the vector bundle itself. The second component $D$ parametrizes sections $h$ supported on the locus $\cM_0^{\geq 1}$ of reducible genus $0$ curves which vanish at the node $n$ and such that the vanishing locus $\VV(h)_n$ localized at the node is either positive dimensional or it has length different from $2$.

\begin{remark}
	The $\gm$-gerbe $\cP_3'\rightarrow \cM_0^{\leq 1}$ is trivial, as there exists a section defined by the association $Z/S \mapsto \omega_{Z/S}^{\otimes 2}$. Therefore, $\cP_3' \simeq \cM_0^{\leq 1}\times \cB\gm$.
\end{remark} 

We have just proved the following description.

\begin{proposition}
	The stack $\Htilde_3 \setminus \Detilde_1$ is isomorphic to the stack $\VV(\pi_*\omega_{\cZ/\cP_3'}^{\otimes -4})\setminus (D \cup 0)$, where $0$ is the zero section of the vector bundle and $D$ parametrizes sections $h$ supported on $\cM_0^1 \times \cB\gm$ such that the vanishing locus $\VV(h)$ localized at the node does not have length $2$.
\end{proposition}

Lastly, we state a fact that is well-known to experts, which is very helpful for our computations. It is one of the reasons why we choose to do the computations inverting $2$ and $3$ in the Chow rings. Consider a morphism $f:\cX \rightarrow \cY$ which is representable, finite and flat of degree $d$ between quotient stacks and consider the cartesian diagram of $f$ with itself: 
$$
 \begin{tikzcd}
\cX \times_{\cY} \cX \arrow[r, "p_1"] \arrow[d, "p_2"] & \cX \arrow[d, "f"] \\
\cX \arrow[r, "f"]                             & \cY.               
\end{tikzcd}
$$
 
\begin{lemma}\label{lem:chow-tor}
	In the situation above, the diagram in the category of groups
	$$ 
	\begin{tikzcd}
	\ch(\cY)[1/d] \arrow[r, "f^*"] & \ch(\cX)[1/d] \arrow[r, "p_1^*"', bend right] \arrow[r, "p_2^*", bend left] & \ch(\cX \times_{\cY} \cX)[1/d]
	\end{tikzcd}
	$$
	is an equalizer.
\end{lemma}

\begin{proof}
	This is just an application of the formula $g_*g^*(-)=d(-)$ for finite flat representable morphisms of degree $d$. 
\end{proof}

\begin{remark}
	
	We are interested in the following application: if $f:\cX \rightarrow \cY$ is $G$-torsor where $G$ is a costant group, we have an honest action of $G$ over $\ch(\cX)$ and the lemma tells us that the pullback $f^*$ is isomorphism between $\ch(\cY)$ and $\ch(\cX)^{G-{\rm inv}}$, if we invert the order of $G$ in the Chow groups.
\end{remark}

\subsection*{Relation coming from the zero-section}
By a standard argument in intersection theory, we know that 
$$ \ch(\VV\setminus 0)=\ch(\cX)/(c_{\rm top}(\VV))$$
for a vector bundle $\VV$ on a stack $\cX$, where $c_{\rm top}(-)$ is the top Chern class. 

To compute the Chern class, first we need to describe the Chow ring of $\cP_3'$. In \cite{EdFul2}, they compute the Chow ring of $\cM_{0}^{\leq 1}$ with integral coefficients, and their result implies that 
$$ \ch(\cP_3')=\ch(\cM_0^{\leq 1}\times \cB\gm)=\ZZ[1/6,c_1,c_2,s]$$
where $c_i:=c_i\Big(\pi_*\omega_{\cZ/\cM_0^{\leq 1}}^{\vee}\Big)$ for $i=1,2$ and $s:=c_1\Big(\pi_*(\cL\otimes\omega_{\cZ/\cM_0^{\leq 1}}^{\otimes -2})\Big)$. We recall that $(\cZ,\cL)$ is the universal object over $\cP_3'$.

\begin{remark}
	We are using the fact that $\ch(\cX\times \cB\gm)\simeq \ch(\cX)\otimes \ch(\cB\gm)$ for a smooth quotient stack $\cX$. This follows from Lemma 2.12 of \cite{Tot}.
\end{remark}

We need to compute $c_{9}(\pi_*\cL^{\otimes -4})$. If we denote by $\cS$ the line bundle on $\cP_3'$ such that $\pi^*(\cS):=\cL\otimes\omega_{\cZ/\cM_0^{\leq 1}}^{\otimes -2}$, we have that
$$ \pi_*\cL^{\otimes -4} = \pi_*\omega_{\cZ}^{\otimes -4} \otimes \cS^{\otimes -2}.$$ 

\begin{proposition}
	We have an exact sequence of vector bundles over $\cP_3'$
	$$
	0\rightarrow \det (\pi_*\omega_{\cZ}^{\vee})\otimes {\rm Sym}^2 \pi_*\omega_{\cZ}^{\vee} \rightarrow {\rm Sym}^4 \pi_*\omega_{\cZ}^{\vee} \rightarrow \pi_*\omega_{\cZ}^{\otimes -4} \rightarrow 0.$$ 
\end{proposition}

\begin{proof}
	Let $S$ be a $\kappa$-scheme and let $(Z/S,L)$ be an object of $\cP_3'(S)$. Consider the $S$-morphism 
	$$
	\begin{tikzcd}
	Z \arrow[rd, "\pi"] \arrow[r, "i", hook] & \PP(\cE) \arrow[d, "p"] \\
	& S                      
	\end{tikzcd}
	$$
	induced by the complete linear system of the line bundle $\omega_{Z/S}^{\vee}$, namely $\cE:=(\pi_*\omega_{Z/S}^{\vee})^{\vee}$. Then $i$ is a closed immersion and we have the following facts:
	\begin{itemize}
		\item $i^*\cO_{\PP(\cE)}(1)=\omega_{Z/S}^{\vee}$,
		\item $\cO_{\PP(\cE)}(-Z)\simeq\omega_{\PP(\cE)/S}(1)$; 
	\end{itemize} 
   see the proof of Proposition 6 of \cite{EdFul2} for a detailed discussion. Because 
   $$ \pi_*\omega_{Z/S}^{\otimes -4}=p_*i_*(\omega_{Z/S}^{\otimes -4})=p_*i^*\cO_{\PP(\cE)}(4)$$
   we can consider the exact sequence
   $$ 0 \rightarrow \cO_{\PP(\cE)}(4-Z) \rightarrow \cO_{\PP(\cE)}(4) \rightarrow i^{*}\cO_{\PP(\cE)}(4) \rightarrow 0. $$ 
   If we do the pushforward through $p$, the sequence remain exact for every geometric fiber over $S$, because $Z$ is embedded as a conic. Therefore we get 
   $$ 0\rightarrow p_*(\cO_{\PP(\cE)}(5)\otimes \omega_{\PP(\cE)/S}) \rightarrow p_*\cO_{\PP(\cE)}(4) \rightarrow \pi_*\omega_{Z}^{\otimes -4} \rightarrow 0$$ 
   and using the formula $\omega_{\PP(\cE)}=\cO_{\PP(\cE)}(-2) \otimes p^*\det \cE^{\vee}$, we get the thesis.
\end{proof}

We have found the first relation in our strata, which is
$$c_9:=\frac{c_{15}(\cS^{\otimes -2}\otimes {\rm Sym}^4 \pi_*\omega_{\cZ}^{\vee} )}{c_6(\cS^{\otimes -2}\otimes\det (\pi_*\omega_{\cZ}^{\vee})\otimes {\rm Sym}^2 \pi_*\omega_{\cZ}^{\vee})}$$
and can be described completely in terms of the Chern classes $c_1,c_2$ of $\pi_*\omega_{\cZ}^{\vee}$ and  $s=c_1(\cS)$.

\subsection*{Relations from the locus $D$}

We concentrate now on the locus $D$. First of all notice that $D$ is contained in the restriction of $\pi_*\cL^{\otimes -4}$ to $\cM_0^{1}\times \cB\gm$.

\begin{remark}
	In \cite{EdFul2}, the authors describe the stack $\cM_0^{\leq 1}$ as the quotient stack $[S/\GL_3]$ where $S$ is an open of the six-dimensional $\GL_3$-representation of homogeneous forms in three variables, namely $x,y,z$, of degree $2$. The action can be described as 
	$$ A.f(x):=\det(A)f(A^{-1}(x,y,z))$$ 
	for every $A \in \GL_3$ and $f \in S$, and the open subscheme $S$ is the complement of the closed invariant subscheme parametrizing non-reduced forms. 
	
	The proof consists in using the line bundle $\pi_*\omega_{\cZ}^{\vee}$, which is very ample, to describe the $\cM_0^{\leq 1}$ as the locus of reduced conics in $\PP^2$ with an action of $\GL_3$. For a more detailed discussion, see \cite{EdFul2}. 
	
	In this setting, $\cM_0^1$ correspond to the closed locus $S^1$ of $S$ parametrizing reducible reduced conics in $\PP^2$. It is easy to see that the action of $\GL_3$ over $S^1$ is transitive, therefore $\cM_0^1 \simeq \cB H$, with $H$ the subgroup of $\GL_3$ defined as the stabilizers of any element, say $xy \in S^1$. A straightforward computation shows that $H\simeq (\gm\ltimes \ga)^2 \ltimes C_2$, where $C_2$ is the costant group with two elements.  
\end{remark}

As we are inverting $2$ in the Chow rings, we can use \Cref{lem:chow-tor} to describe the Chow ring of $\cM_0^1$ as the invariant subring of $\ch(\cB\gm^2)$ of a specific action of $C_2$. The $\ga$'s do not appear in the computation of the Chow ring thanks to Proposition 2.3 of \cite{MolVis}. We can see that the elements of the form $(t_1,t_2,1)$ of $\gm^2 \ltimes C_2$ correspond to the matrices in $\GL_3$ of the form 
$$
\begin{pmatrix*}
	t_1 & 0 & 0 \\
	0 & t_2 & 0 \\
	0 & 0   &  1  
\end{pmatrix*};
$$
 the elements of the form $(t_1,t_2,-1)$ correspond to the matrices 
$$ 
\begin{pmatrix*}
0 & t_1 & 0 \\
t_2 & 0 & 0 \\
0 & 0   &  1  
\end{pmatrix*}.
$$
It is immediate to see that the action of $C_2$ over $\gm^2$ can be described as $(-1).(t_1,t_2)=(t_2,t_1)$. Therefore if we denote by $t_1$ and $t_2$ the generator of the Chow ring of the two copies of $\gm$ respectively, we get  
$$ \ch(\cB(\gm^2\ltimes C_2)) \simeq \ZZ[1/6, t_1+t_2, t_1t_2].$$ 

A standard computation shows the following result.

\begin{lemma}\label{lem:normal-xi}
	If we denote by $i:\cM_0^1 \into \cM_0^{\leq 1}$ the (regular) closed immersion, we have that $i^{*}(c_1)=t_1+t_2$ and $i^{*}(c_2)=t_1t_2$ and therefore $i^{*}$ is surjective. Moreover, we have the equality $[\cM_0^1]=-c_1$ in the Chow ring of $\cM_0^{\leq 1}$. 
\end{lemma}

\begin{proof}
	The description of $i^{*}(c_i)$ for $i=1,2$ follows from the explicit description of the inclusion 
	$$\cM_0^1 = [\{xy\}/H] \into [S/\GL_3]=\cM_0^{\leq 1}.$$
	
	Regarding the second part of the statement, it is enough to observe that $\cM_0^1=[S^1/\GL_3] \into [S/\GL_3]$ where $S^1$ is the hypersurface of $S$ described by the vanishing of the determinant of the general conic. A straightforward computation of the $\GL_3$-character associated to the determinant formula shows the result.
	\end{proof}
Finally, we focus on $D$. The vector bundle $\pi_*\cL^{\otimes -4}$ (or equivalently $\pi_*\omega_{\cZ}^{\otimes -4} \otimes \cS^{\otimes -2}$) can now be seen as a $9$-dimensional $H$-representation. Specifically, we are looking at sections of $\pi_*\omega_{\cZ}^{\otimes -4} \otimes \cS^{\otimes -2}$ on the curve $xy=0$, which are a $9$-dimensional vector space $\AA(4,4)$ parametrizing a pair of binary forms of degree $4$, which have to coincides in the point $x=y=0$. Let us denote by $\infty$ the point $x=y=0$, which is in common for the two components. With this notation, $D$ parametrizes pairs $(f(x,z),g(y,z))$ such that $f(\infty)=g(\infty)=0$ and either the coefficient of $xz^3$ or the one of $yz^3$ vanishes. 

\begin{remark}
	This follows from the local description of the $2:1$-cover. In fact, if $\infty$ is the intersection of the two components, we have that \'etale locally the double cover looks like
	$$ k[[x,y]]/(xy) \into k[[x,y,t]](t^2-h(x,y))$$ 
	where $h$ is exactly the section of $\pi_*\omega_{\cZ}^{\otimes -4} \otimes \cS^{\otimes -2}$. Because we can only allow nodes or tacnodes as fibers over a node in the quotient morphism by Proposition 3.2 of \cite{Per2}, we get that $h$ is either a unit (the quotient morphism is \'etale) or $h$ is of the form $xp(x)+yq(y)$ such that $p(0)\neq 0$ and $q(0)\neq 0$.  
\end{remark}

The action of $\gm^2\times \gm$ (where the second group of the product is the one whose generator is $s$) over the coefficient of $x^iz^{4-i}$ (respectively $y^iz^{4-i}$) can be described by the character $t_1^is^{-2}$ (respectively $t_2^is^{-2}$).

\begin{lemma}
	The ideal of relations coming from $D$ in $\VV(\pi_*\cL^{\otimes -2})\vert_{\cM_0^1 \times \cB\gm}$ is generated by the two classes $2s(4s-(t_1+t_2))$ and $2s(4s^2-2s(t_1+t_2)+t_1t_2))$. Therefore we have that the ideal of relations coming from $D$ in $\cP_3'$ is generated by the two relations $D_1:=2sc_1(c_1-4s)$ and $D_2:=2sc_1(4s^2-2sc_1+c_2)$. 
\end{lemma}

\begin{proof}
	Because of \Cref{lem:chow-tor}, we can start by computing the ideal of relations in the $\gm^2$-equivariant setting (i.e. forgetting the action of $C_2$) and then considering the invariant elements (by the action of $C_2$). It is clear the ideal of relation $I$ in the $\gm^2$-equivariant setting is of the form $(2s(2s-t_1),2s(2s-t_2))$. Thus the ideal $I^{\rm inv}$ is generated by the elements $2s(4s-(t_1+t_2))$ and $2s(2s-t_1)(2s-t_2)$.
\end{proof}

As a corollary, we get the Chow ring of $\Htilde_3\setminus \Detilde_1$. Before describing it, we want to change generators. We can express $c_1$, $c_2$ and $s$ using the classes $\lambda_1$, $\lambda_2$ and $\xi_1$ where $\lambda_i$ as usual is the $i$-th Chern class of the Hodge bundle $\HH$ and $\xi_1$ is the fundamental class of $\Xi_1$, which is defined as the pullback
$$
\begin{tikzcd}
\Xi_1 \arrow[d] \arrow[r, hook]      & \Htilde_3\setminus \Detilde_1 \arrow[d] \\
\cM_0^1\times \cB\gm \arrow[r, hook] & \cP_3'=\cM_0^{\leq 1}\times \cB\gm    
\end{tikzcd}
$$
and it parametrizes hyperelliptic curves without separating nodes such that the geometric quotient has two irreducible components.

\begin{lemma}\label{lem:lambda-class-H}
	In the setting above, we have that $s=(-\xi_1 - \lambda_1)/3$, $c_1=-\xi_1$ and $c_2= \lambda_2 - (\lambda_1^2 - \xi_1^2)/3$. Furthermore, we have the following relation
	$$\lambda_3=\frac{(\xi_1+\lambda_1)(9\lambda_2+(\xi_1+\lambda_1)(\xi_1-2\lambda_1))}{27}.$$
\end{lemma}

\begin{proof}
	First of all, the relation $\xi_1=-c_1$ is clear from the construction of $\xi_1$, as we have already computed the fundamental class $\cM_0^1$ in $\ch(\cM_0^{\leq 1})$.
	
	Let $f:C\rightarrow Z$ be the quotient morphism of an object in $\Htilde_3\setminus \Detilde_1$ and let $\pi_C:C\rightarrow S$ and $\pi_Z:Z\rightarrow S$ be the two structural morphisms. Grothendieck duality implies that 
	$$ f_*\omega_{C/S} = \curshom_Z(f_*\cO_C, \omega_{Z/S})$$ 
	but because $f$ is finite flat, we know that $f_*\cO_C=\cO_Z \oplus L$, i.e. $f_*\omega_{C}=\omega_{Z}\oplus (\omega_{Z}\otimes L^{\vee})$. Recall that $L\simeq \omega_{Z/S}^{\otimes 2} \otimes \pi_Z^*\cS$ for a line bundle $\cS$ on the base. Therefore if we consider the pushforward through $\pi_Z$, we get
$$ \pi_{C,*}\omega_{C/S}=\pi_{Z,*}(\omega_{Z/S}^{\vee})\otimes \cS^{\vee}$$
and the formulas in the statement follow from simple computations with Chern classes.
 \end{proof}
\begin{corollary}\label{cor:chow-hyper}
	We have the following isomorphism of rings:
	$$\ch(\Htilde_3)= \ZZ[1/6,\lambda_1,\lambda_2,\xi_1]/(c_9,D_1,D_2)$$
	where $D_1=2\xi_1(\lambda_1+\xi_1)(4\lambda_1+\xi_1)/9$, $D_2:=2\xi_1(\xi_1+\lambda_1)(9\lambda_2+(\xi_1+\lambda_1)^2)/27$ and $c_9$  is a homogeneous polynomial of degree $9$.
\end{corollary}

\begin{remark}
    The polynomial $c_9$ has the following form:
    
    \begin{equation*}
    \begin{split}
    c_9 = & -\frac{16192}{19683}\lambda_1^9 - \frac{23200}{6561}\lambda_1^8\xi_1 - \frac{31040}{6561}\lambda_1^7\xi_1^2 +
    \frac{1376}{729}\lambda_1^7\lambda_2 - \frac{320}{6561}\lambda_1^6\xi_1^3 + \\& + \frac{4576}{243}\lambda_1^6\xi_1\lambda_2 +
    \frac{30784}{6561}\lambda_1^5\xi_1^4  + \frac{10144}{243}\lambda_1^5\xi_1^2\lambda_2 + \frac{3968}{81}\lambda_1^5\lambda_2^2 +
    \frac{16256}{6561}\lambda_1^4\xi_1^5 +  \\ & + \frac{15136}{729}\lambda_1^4\xi_1^3\lambda_2 + \frac{992}{27}\lambda_1^4\xi_1\lambda_2^2
    - \frac{320}{243}\lambda_1^3\xi_1^6 - \frac{5792}{243}\lambda_1^3\xi_1^4\lambda_2 -
    \frac{11072}{81}\lambda_1^3\xi_1^2\lambda_2^2 - \\ & - \frac{7264}{27}\lambda_1^3\lambda_2^3 - \frac{7360}{6561}\lambda_1^2\xi_1^7  -
    \frac{5216}{243}\lambda_1^2\xi_1^5\lambda_2 - \frac{11392}{81}\lambda_1^2\xi_1^3\lambda_2^2 -
    \frac{2848}{9}\lambda_1^2\xi_1\lambda_2^3 + \\ & + \frac{640}{6561}\lambda_1\xi_1^8 +  \frac{1952}{729}\lambda_1\xi_1^6\lambda_2 +
   \frac{832}{27}\lambda_1\xi_1^4\lambda_2^2 + \frac{1568}{9}\lambda_1\xi_1^2\lambda_2^3 +384\lambda_1\lambda_2^4 + \\ & +
    \frac{2912}{19683}\xi_1^9 + \frac{352}{81}\xi_1^7\lambda_2 + \frac{3808}{81}\xi_1^5\lambda_2^2 +
    \frac{5984}{27}\xi_1^3\lambda_2^3 + 384\xi_1\lambda_2^4.
    \end{split}
    \end{equation*}
    
\end{remark}

\subsection*{Normal bundle of $\Htilde_3\setminus \Detilde_1$ in $\Mtilde_3 \setminus \Detilde_1$}
We end up the section with the computation of the first Chern class of the normal bundle of the closed immersion $\Htilde_3\setminus \Detilde \into \Mtilde_3 \setminus \Detilde_1$. For the sake of notation, we denote the normal bundle by $N_{\cH|\cM}$.

\begin{proposition}
	The fundamental class of  $\Hbar_3$  in $\ch(\Mbar_3)$ is equal to $9\lambda_1-\delta_0-3\delta_1$.
\end{proposition}
\begin{proof}
	This is Theorem 1 of \cite{Est}. It is important to notice that in the computations the author just need to invert $2$ in the Picard group to get the result.
\end{proof}

\begin{remark}
	As in the ($A_1$-)stable case, we define by $\Detilde_0$ the closure of the substack of $\Mtilde_3$ which parametrizes curves with a non-separating node. Alternately, we can consider the stack $\Detilde$ in the universal curve $\Ctilde_3$ of $\Mtilde_3$, defined as the vanishing locus of the first Fitting ideal of $\Omega_{\Ctilde_3|\Mtilde_3}$.
	
	We denote by $\Detilde$ the image with its natural stacky structure and by $\Detilde_0$ the complement of the inclusion $\Detilde_1 \subset \Detilde$. Thanks to Lemma 4.7 of \cite{Per1}, we know that $\Detilde_1\into \Detilde$ is also open, therefore we get that $\Detilde_0$ is a closed substack of $\Mtilde_3$. We denote by $\delta_0$ its fundamental class in the Chow ring of $\Mtilde_3$.
\end{remark}

Because $\Mtilde_3\setminus\Mbar_3$ has codimension $2$, we get that the same formula works in our context. Because $\delta_1$ is defined as the fundamental class of $\Detilde_1$, we just need to compute $\delta_0$ restricted to $\Htilde_3\setminus \Detilde_1$ to get the description we want. To do so, we compute the restriction of $\delta_0$ to $\Htilde_3 \setminus (\Detilde_1 \cup \Xi_1)$ and to $\Xi_1\setminus \Detilde_1$ and then glue the informations together.

First of all, notice that $\Htilde_3 \setminus (\Detilde_1 \cup \Xi_1)$ is an open inside a $9$-dimensional representation $V$ of $\PGL_2\times \gm$, as we have $2:1$-covers of $\PP^1$. Unwinding the definitions, we get the following result.

\begin{lemma}\label{lem:ar-vis}
	The representation $V$ of $\PGL_2 \times \gm$ above coincides with the one given by Arsie and Vistoli in Corollary 4.6 of \cite{ArVis}.  
\end{lemma} 
 
This implies that we can see it as an open inside $[\AA(8)/(\GL_2/\mu_4)]$, where $\AA(8)$ is the vector space of binary forms of degree $8$ and $\GL_2/\mu_4$ acts by the equation $A.f(x)=f(A^{-1}x)$. By the theory developed in \cite{ArVis} (see \Cref{prop:descr-hyper} in our situation), it is clear that the sections $f \in \AA(8)$  describe the branching locus of the quotient morphism. In particular, worse-than-nodal singularities on the $2:1$-cover of the projective line correspond to points on $\PP^1$ where the branching divisor is not \'etale, or equivalently points where $f$ has multiplicity more than $1$. Therefore $\delta_0$ is represented by the closed invariant subscheme of singular forms inside $\AA(8)$. This was already computed by Di Lorenzo (see the first relation in Theorem 6 in \cite{DiLor}), and we have that $ \delta_0=28\tau$ with $$\tau=c_1(\pi_*(\omega_C(-W)^{\otimes 2}))$$
where $W$ is the ramification divisor in $C$. Notice that if $f:C\rightarrow \PP^1$ is the cyclic cover of the projective line, we have that $W\simeq f^{*}\cL^{\otimes \vee}$. A computation using Grothendieck duality gives us that $\tau=-s$.

 We have that $\delta_0:=as+bc_1$ in $\ch(\Htilde_3\setminus \Detilde_1)$ for some elements in $\ZZ[1/6]$. 
 
 The computations above implies that if we restrict to the open complement of $\Xi_1 \setminus \Detilde_1$, we get $a=-28$.

The restriction to $\Xi_1\setminus \Detilde_1$ is a bit more complicated, because we have that $\Xi_1 \subset \Detilde_0$. Recall the description
$$ \Xi_1\setminus \Detilde_1 = [\AA(4,4)\setminus D/H]$$
where $\AA(4,4)$ is the vector space of pairs of binary forms $(f(x,z),g(y,z))$ of degree $4$ such that $f(0,1)=g(0,1)$. We define an open $\Xi_1^0$ of $\Xi_1\setminus \Detilde_1$ which are the pairs $(f,g)$ such that $f(0,1)=g(0,1)\neq 0$. Clearly, $D$ does not intersect $\Xi_1^0$.

We do the computations on $\Xi_1^0$ and verify they are enough to determine the coefficient $b$ in the description $\delta_0=-28s+bc_1$. 

\begin{remark}
	The class $[\Xi_1\setminus \Xi_1^0]$ in $\ch(\Xi_1)$ is equal to $-2s$. In fact, it can be described as the vanishing locus of the coefficient of $z^4$ for the pair $(f,g) \in \AA(4,4)$. Therefore the Picard group (up to invert $2$) of $\Xi_1^0$ is freely generated by $c_1$.
\end{remark}

Let us define a closed substack inside $\Xi_1^0$: we define $\Delta'$ as the locus parametrizing pairs $(f,g)$ such that either $f$ or $g$ are singular forms.

\begin{lemma}
	In the setting above, we have the equality 
	$$ \Delta'=12c_1$$ 
	in the Picard group of $\Xi_1^0$.
\end{lemma}

\begin{proof}
	As a conseguence of \Cref{lem:chow-tor}, we can do the $\gm^2$-equivariant computations of the equivariant class of $\Delta'$. We have that $\Delta'=\Delta_1' \cup \Delta_2'$ where $\Delta_1'$ (respectively $\Delta'_2$) is the substack parametrizing pairs $(f,g)$ such that $f$ (respectively $g$) is a singular form. 

We reduce ourself to compute the class the locus of singular forms inside $\AA(4)$. The result then follows from a straightforward computation. 
\end{proof}
Now we are ready to compute the restriction of $\delta_0$.
\begin{lemma}
	In the situation above, we have 
	$$ \delta_0\vert_{\Xi_1^0}= -2c_1 + [\Delta'] $$
	inside the Chow ring of $\Xi_1^0$.
\end{lemma}

\begin{proof}
	Because we are computing the Chern class of a line bundle, we can work up to codimension two, and in particular we can remove the locus parametrizing curves with $A_r$-singularities for $r\geq 2$  (this locus intersects $\Xi_1^0$ transversely). Therefore, it is enough to prove the statement in the stable case. For the sake of notation, let $\cM$ be $\Mbar_3\setminus \Debar_1$, $\cH$ be $\Hbar_3\setminus \Debar_1$ and $\cC$ (respectively $\cC_{\cH}$) be the universal curve over $\cM$ (respectively $\cH$). We denote by $\Xi$ the open substack of $\Xi_0^1$ classifying $A_1$-stable curves.
	
	Consider the closed substack $\Delta$ in $\cC$ of singular points of the morphism $\pi:\cC \rightarrow \cM$, which can be defined as the vanishing of the first Fitting ideal of $\Omega_{\cC/\cM}$. In the same way, one can define $\Delta_{\cH}$ inside $\cC_{\cH}$. Notice that the intersection of $\Detilde_0$ with $\cM$ is exactly the image of $\Delta$ through the morphism $\pi$, and $\pi\vert_{\Delta}: \Delta \rightarrow \cM$ is birational onto its image. 
	In the following diagram
	$$
	\begin{tikzcd}
                   & \Delta_ {\cH} \arrow[d, "p"] \arrow[r, "\iota", hook] & \Delta \arrow[d] \\
    \Xi \arrow[r, "j"] & \cH \arrow[r, "i", hook]                                                    & \cM             
    \end{tikzcd}
    $$
    we have that the square is cartesian and the two closed immersion are regular of codimension $1$. This implies that 
    $$ \delta_0 \vert_{\Xi} = j^*p_*(1).$$
    
    We need to describe the fiber product $\Delta_{\cH}\times_{\cH} \Xi$. The moduli stack $\Xi$ parametrizes triples $(C,\sigma,E)$ where $(C,\sigma)$ is a stable hyperelliptic curve of genus $3$ and $E$ is a finite \'etale divisor of degree $2$ of $C$ (supported in the singular locus) such that the partial normalization of $C$ in $E$ are two disjoint curves of genus $1$ and $\sigma$ acts non-trivially on $E$. Therefore the fiber product $\Delta_{\cH}\times_{\cH} \Xi$ parametrizes $(C,\sigma,E,p)$ where $(C,\sigma,E)$ is an object of $\Xi$ and $p$ is a node of $C$. We denote by $D'$ the (open) substack of the fiber product where $p$ is not contained in the support of $E$. It is easy to see that this is also a closed substack. If we denote by $D''$ the complement of $D'$ inside the fiber product, we have the following cartesian diagram:
    $$
    \begin{tikzcd}
    D' \sqcup D'' \arrow[d, "q'\sqcup q''"'] \arrow[r, "j'\sqcup j''"] & \Delta_ {\cH} \arrow[d, "p"] \\
    \Xi \arrow[r, "j"]                                                 &    \cH                         
    \end{tikzcd}$$
    where $j':D' \hookrightarrow \Delta_{\cH}$ is a regular embedding of codimension $1$ and $j'':D'' \hookrightarrow \Delta_{\cH}$ is a regular embedding of codimension $0$, i.e. an open and closed embedding. Therefore $$\delta_0\vert_{\Xi} = q''_*c_1(q''^*N_j)+q'_*(1)$$ where $N_j$ is the normal bundle associated to the closed embedding $j:\Xi \hookrightarrow \cH$. Notice that $q''$ is finite of generic degree $2$, therefore \Cref{lem:normal-xi} implies that 
    $$ q''_*c_1(q''^*N_j)=2c_1(N_j)=-2c_1$$ 
    as an element of the Picard group of $\Xi$ (which is equal to the one of $\Xi_0^1$). The statement follows from noticing that the image of $q'$ is exactly $\Delta'$ and that $q'$ is birational onto its image.

\end{proof}

The previous lemma implies that $\delta_0 \vert_{\Xi_1^0}=10c_1$ which implies that $b=10$, i.e. $\delta_0=-28s+10c_1$ in the Picard group of $\Htilde_3\setminus \Detilde_1$. \Cref{lem:lambda-class-H} implies the following alternative description.

\begin{corollary}\label{cor:norm-hyper}
	The first Chern class of $N_{\cH|\cM}$ is equal to $(2\xi_1-\lambda_1)/3$.
\end{corollary}
 
\section{Description of $\Mtilde_3 \setminus (\Detilde_1\cup \Htilde_3)$}\label{sec:3}

We focus now on the open stratum. Recall that the canonical bundle of a smooth curve of genus $g$ is either very ample or the curve is hyperelliptic and the quotient morphism factors through the canonical morphism. This cannot be true for $A_r$-stable curves as in $\Detilde_1$ we have that the dualizing line bundle is not globally generated, see Proposition 3.36 of \cite{Per2}. Nevertheless, if we remove $\Detilde_1$, we have the same result for genus $3$ curves. 

\begin{lemma}
	Suppose $C$ is an $A_r$-stable curve of genus $3$ over an algebraically closed field which does not have separating nodes and it is not hyperelliptic. Then the canonical morphism (induced by the complete linear system of the dualizing sheaf) is a closed immersion.
\end{lemma}

\begin{proof}
	This proof is done using the theory developed in \cite{Cat} to deal with most of the cases and analizying the rest of the them separately. 
	
	Firstly, we prove that if $C$ is a $2$-connected $A_r$-stable curve of genus $3$ (i.e. it is in $\Mtilde_3\setminus \Detilde_1$) then we have that $C$ is hyperelliptic if and only if there exist two smooth points $x,y$ such that $\dim \H^0(C,\cO(x+y))=2$ (notice that this is the definition of hyperelliptic as in \cite{Cat}).
	
	One implication is clear. Suppose there exists two smooth points $x,y$ in C such that $\dim \H^0(C,\cO(x+y))=2$. Proposition 3.14 of \cite{Cat} assures us that we are in one of two possible situations:
	\begin{itemize}
		\item[(a)] $x,y$ belongs  to $2$ different irreducible components $Y_1,Y_2$ of genus $0$ such that every connected component $Z$ of $C-Y_1-Y_2$ intersect $Y_1$ in a node and $Y_2$ in a (different) node,
		\item[(b)] $x,y$ belong to an irreducible hyperelliptic curve $Y$ such that for every connected component $Z$ of $C-Y$ intersect $Y$ in a Cartier divisor isomorphic to $\cO(x+y)$.
	\end{itemize} 
	Regarding (a), the stability condition implies that the only possibilities are either that there are no other connected components of $C-Y_1-Y_2$, which implies $C$ is hyperelliptic, or we have only one connected component $Z$ of $C-Y_1-Y_2$ which is of genus $1$. Because $Z$ is of genus $1$ and it intersects $Y_1\cup Y_2$  in two points, we have that there exists a unique hyperelliptic involution of $Z$ that exchanges them, see Lemma 3.6 of \cite{Per2}. Therefore again $C$ is hyperelliptic. In case (b), the stability condition implies that the only possibilities are either that $C$ is irreducible, and therefore hyperelliptic, or $C$ is the union of two genus $1$ curves intersecting in a length $2$ divisor. Again it follows from Lemma 3.6 of \cite{Per2} that $C$ is hyperelliptic.
	
	Now, we focus on an other definition given in \cite{Cat}. The author define $C$ to be strongly connected if there are no pairs of nodes $x,y$ such that $C\setminus \{x,y\}$ is disconnected. Furthermore, the author define $C$ very strongly connected if it is strongly connected and there is not a point $p \in C$ such that $C\setminus \{p\}$ is disconnected.  
	
	In our situation, a curve $C$ is not very strongly connected if 
	\begin{itemize}
		\item[(1)] $C$ is the union of two genus $1$ curves meeting at a divisor of length $2$,
		\item[(2)] $C$ is the union of two genus $0$ curves meeting in a singularity of type $A_7$,
		\item[(3)] $C$ is the union of a genus $0$ and a genus $1$ curve meeting in a singularity of type $A_5$.
	\end{itemize}
	
	Case (1) is always hyperelliptic. An easy computation shows that the case (3) is never hyperelliptic and the canonical morphism is a closed embedding which identifies $C$ with the union of a cubic and a flex tangent in $\PP^2$. Finally, one can show that in case (2) the canonical morphism restricted to the two components is a closed embedding, therefore it is clear that it is either a finite flat morphism of degree $2$ over its image ($C$ hyperelliptic) or it is a closed immersion globally on $C$. 
	
	It remains to prove the statement in the case $C$ is very strongly connected. This is Theorem G of \cite{Cat}.
\end{proof}

\begin{remark}
	Notice that this lemma is really specific to genus $3$ curves and it is false in genus $4$. Consider a genus $2$ smooth curve $C$ meeting a genus $1$ smooth curve $E$ in two points, which are not a $g_1^2$ for $C$. Then the canonical morphism is $2:1$ restricted to $E$ but it is birational on $C$.
\end{remark}

The previous lemma implies that the description of $\cM_3 \setminus \cH_3$ proved by Di Lorenzo in Proposition 3.1.3 of \cite{DiLor2} can be generalized in our setting. Specifically, we have the following isomorphism: 
$$ \Mtilde_3 \setminus (\Detilde_1\cup \Htilde_3) \simeq [U/\GL_3]$$ 
where $U$ is an invariant open subscheme inside the space $\AA(3,4)$ of (homogeneous) forms in three coordinates of degree $4$ which is a representation of $\GL_3$ with the action described by the formula $A.f(x):=\det(A) f(A^{-1}x)$. The complement parametrizes forms $f$ such that the induced projective curve $\VV(f)$ in $\PP^2$ is not $A_r$-prestable. 

We use the description as a quotient stack to compute its Chow ring. The strategy is similar to the one adopted in \cite{DiLorFulVis} with a new idea to simplify computations. 

As usual, we pass to the projectivization of $\AA(3,4)$ which we denote by $\PP^{14}$. We induce an action of $\GL_3$ on $\PP^{14}$ setting $A.[f]=[f(A^{-1}x)]$, and if we denote by $\overline{U}$ the projectivization of $U$, we get 
$$\ch_{\GL_3}(U)= \ch_{\GL_3}(\overline{U})/(c_1-h)$$
where $c_i$ is the $i$-th Chern class of the standard representation of $\GL_3$ and $h=\cO_{\PP^{14}}(1)$ the hyperplane section of the projective space of ternary forms. The idea is to compute the relations that come from the closed complement of $\overline{U}$ and then set $h=c_1$ to get the Chow ring of $\Mtilde_3\setminus (\Htilde_3\cup \Detilde_1)$ as the quotient of $\ch(\cB\GL_3)$ by these relations.

\begin{remark}
	Notice that $\lambda_i:=c_i(\HH)$ where $\HH$ is the Hodge bundle can be identified with the Chern classes of the dual of the standard representation.
\end{remark}

We consider the quotient (stack) of $\PP^{14}\times \PP^{2}$ by the following $\GL_3$-action
$$A.([f],[p]):=([f(A^{-1}x),Ap]),$$
and we denote by $Q_4$ the universal quartic over $[\PP^{14}/\GL_3]$, or equivalently the substack of $[\PP^{14}\times \PP^2/\GL_3]$ parametrizing pairs $([f],[p])$ such that $f(p)=0$.

Now, we introduce a slightly more general definition of $A_n$-singularity.

\begin{definition}\label{def:A-sing}
	We say that a point $p$ of a curve $C$ is an $A_{\infty}$-singularity if we have an isomorphism 
$$ \widehat{\cO}_{C,p}\simeq k[[x,y]]/(y^2).$$
Furthermore, we say that $p$ is a $A$-singularity if it an $A_n$-singularity for $n$ either a positive integer or $\infty$.
\end{definition}

We describe when an $A_{\infty}$-singularity can occur for plane curves.

\begin{lemma}
	A point $p$ of a plane curve $f$ is an $A_{\infty}$-singularity if and only if $p$ lies on a unique irreducible component $g$ of $f$ where $g$ is the square of a smooth plane curve.
\end{lemma}

\begin{proof}
	Denote by $A$ the localization of $k[x,y]/(f)$ at the point $p$, which we can suppose to be the maximal ideal $(x,y)$. Because $A$ is an excellent ring and the completion is non-reduced, we get that $A$ is also non reduced. Let $h$ be a nilpotent element in $A$. Because the square of the nilpotent ideal in the completion is zero, we get that $h^2=0$. Since $k[x,y]$ is a UFD, we get the thesis.
\end{proof}

The reason why we introduced $A$-singularity is that they have an explicit description in terms of derivative of the defining equation for plane curves.

\begin{lemma}\label{lem:A-sing}
	A point $p$ on a plane curve defined by $f$ is not an $A$-singularity if and only if both the gradient and the Hessian of $f$ vanishes at the point $p$.
\end{lemma}

\begin{proof}
	If $f$ is an $A$-singularity, one can compute its Hessian and gradient (up to a change of coordinates) looking at the complete local ring, therefore it is a trivial computation.
	
	On the contrary, if the gradient does not vanish, it is clear that $p$ is a smooth point of $f=0$. Otherwise, if the gradient vanishes but there is a double derivative different from zero, we can use Weierstrass preparation theorem and the square completion procedure ($ {\rm char}(\kappa)\neq 2$) to get the result. 
\end{proof}

Now, we introduce a weaker definition of Chow envelopes, which depends on what coefficients we consider for the Chow groups.

\begin{definition}
	Let $R$ be a ring. We say that a representable proper morphism $f:X\rightarrow Y$ is an algebraic Chow envelope for $Y$ with coefficients in $R$  if the morphism $f_*:\ch(X)\otimes_{\ZZ}R \rightarrow \ch(Y)\otimes_{\ZZ}R$ is surjective. 
\end{definition}

\begin{remark}
Recall the definition of Chow envelope between algebraic stacks as in Definition 3.4 of \cite{DiLorPerVis}. Because they are working with integer coefficients, Proposition 3.5 of \cite{DiLorPerVis} implies that an algebraic Chow envelope is an algebraic Chow envelope for every choice of coefficients.
\end{remark}

From now on, algebraic Chow envelopes with coefficients in $\ZZ[1/6]$ are simply called algebraic Chow envelopes.

Consider now the substack $X\subset Q_4$ parametrizing pairs $(f,p)$ such that $p$ is singular but not an $A$-singularity of $f$ . Thanks to the previous lemma, we can describe $X$ as the vanishing locus of the second derivates of $f$ in $p$. By construction, $X$ cannot be an algebraic Chow envelope for the whole complement of $\overline{U}$ because we have quartics $f$ which are squares of smooth conics and they do not appear in $X$. Therefore, we will also add the relations coming from the locus parametrizing squares of conics.

Firstly, we want to study the fibers of the proper morphism $X\rightarrow \PP^{14}$ to prove that it is an algebraic Chow envelope for its image. We need to understand how many and what kind of singular point appears in a reduced quartic in $\PP^2$.

\begin{lemma}\label{lem:sing-points}
	A reduced quartic in $\PP^2$ has at most $6$ singular points. If it has exactly $5$ singular points, then it is the union of a smooth conic and a reducible reduced one. If it has exactly $6$ singular points, then it is the union of two reducible reduced conics that do not share a component.
\end{lemma}

\begin{proof}
 	Suppose the quartic $F:=\VV(f)$ is irreducible. Then we can have at most $3$ singular points. In fact, suppose $p_1,\dots,p_4$ are four singular points. Then there exists a conic $Q$ passing through the four points and another smooth point of $f$. Thus $Q \cap F$ would have length at least $9$, which is impossible by Bezout's theorem. 
 	
 	The same reasoning cannot apply if $F$ is the union of two smooth conics meeting at four points, which is a possible situation. Nevertheless, if we suppose that $F$ has at least $5$ different singular points we would have that there exists a conic $Q$ inside $F$, therefore $F=Q\cup Q'$ with $Q'$ another conic because $F$ is a quartic. It is then clear that the singular points are at most $6$ and one can prove the rest of the statement case by case.
\end{proof}

 We denote by $Z_{\{2\}}^{[2]}$ the substack of quartics which are squares of smooth conics and by $\overline{Z}_{\{2\}}^{[2]}$ its closure in $\PP^{14}$. We use the same notation as in \cite{DiLorFulVis}.

Let us denote by $z_2$ the fundamental class of $\overline{Z}_{\{2\}}^{[2]}$ in $\ch_{\GL_3}(\PP^{14})$. We also denote by $\rho$ the morphism $X\rightarrow \PP^{14}$and by $i_T:T\into \PP^{14}$ the closed complement of $\overline{U}$ in $\PP^{14}$. We are ready for the main proposition.

\begin{proposition}
	The ideal generated by $\im{i_{T,*}}$ is equal to the ideal generated by $\im{\rho_{*}}$ and by $z_2$.
\end{proposition}

\begin{proof}
 Consider $\PP^5$ the space of conics in $\PP^2$ with the action of $\GL_3$ defined by the formula $A.f(x):=f(A^{-1}x)$ and the equivariant morphism $\beta:\PP^5 \rightarrow \PP^{14}$ defined by the association $f \mapsto f^2$. We are going to prove that 
 $$ \rho \sqcup \beta : X \sqcup \PP^5 \longrightarrow \PP^{14}$$
 is an algebraic Chow envelope for $T$ and that the only generator of the image of $\beta_*$ is the fundamental class $\beta_*(1)$, which coincides with $z_2$.
 
 Let $L/\kappa$ be a field extension. First of all, \Cref{lem:sing-points} tells us that a reduced quartic in $\PP^2$ has at most $6$ singular points. Therefore if $f$ is an $L$-point of $\PP^{14}$ which represents a reduced quartic, the fiber $\rho^{-1}(f)\rightarrow \spec L$ is a finite flat morphism of degree at most $6$. As we are inverting $2$ and $3$ in the Chow rings, the only case we need to worry is when the morphism has degree $5$, i.e. when $f$ is union of a singular conic and a smooth conic. However, in that situation we have a rational point, namely the section that goes to the intersection of the two lines that form the singular conic. This prove that $\rho$ is an algebraic Chow envelope for the open of reduced curves in $T$.
 
 Consider a $L$-point $f$ of $\PP^{14}$ which represents a non-reduced quartic. Then $f$ is one of the following:
 \begin{enumerate}
 	\item[(1)] $f$ is the product of a double line and a reduced conic that does not contain the line,
 	\item[(2)] $f$ is the product of a triple line and a different line,
 	\item[(3)] $f$ is the product of two different double lines,
 	\item[(4)] $f$ is the fourth power of a line,
  	\item[(5)] $f$ is a double smooth conic.
 \end{enumerate}
For a more detailed discussion, see Section 1 of \cite{DiLorFulVis}.

We are going to prove that in situations from (1) to (4), the fiber $\rho^{-1}(f)$ is an algebraic Chow envelope for $\spec L$. In cases (1) and (3), we have that the fiber is finite of degree respectively 2 and 1, therefore an algebraic Chow envelope. In cases (2) and (4) the fiber is a line, which is a projective bundle and therefore an algebraic Chow envelope (we do not have to worry about non-reduced structures).

Clearly the fiber $\rho^{-1}(f)$ in the case (5) is empty, as $f$ has only $A_{\infty}$-singularities as closed points. Therefore we really need the morphism $\beta$ which is an algebraic Chow envelope for the case (5). 

It remains to prove the image of $\beta_*$ is generated by $\beta_*(1)$. This follows from the fact that $\beta^{*}(h_{14})=2h_{5}$, where $h_{14}$ (respectively $h_5$) is the hyperplane section of $\PP^{14}$ (respectively $\PP^5$).

\end{proof}

We conclude this section computing explicitly the relations we need and finally getting the Chow ring of $\Mtilde_3\setminus (\Detilde_1 \cup \Htilde_3)$.

The computation for the class $z_2$ can be done using an explicit localization formula. As a matter of fact, the exact computation was already done in Proposition 4.6 of \cite{DiLorFulVis}, although it was not shown as it was not relevant for their computations.

\begin{remark}
	After the identification $h=-\lambda_1$, the localization formula gives us
	\begin{equation*}
	\begin{split}
	z_2=-1152\lambda_1^3\lambda_3^2 + 256\lambda_1^2\lambda_2^2\lambda_3 + 5824\lambda_1\lambda_2\lambda_3^2 - 1152\lambda_2^3\lambda_3 - 10976\lambda_3^3.
	\end{split}
	\end{equation*}
\end{remark}

In order to compute the ideal generated by $\im{\rho_*}$, we introduce a simplified description of $Q_4$ and of $X$.

\begin{lemma}
	The universal quartic $Q_4$ is naturally isomorphic to the quotient stack $[\PP^{13}/H]$ where $H$ is a parabolic subgroup of $\GL_3$.
\end{lemma}

\begin{proof}
	The action of $\GL_3$ is transitive over $\PP^2$, therefore we can consider the subscheme $\PP^{14}\times \{[0:0:1]\}$ in $\PP^{14}\times \PP^2$. If we denote by $H$ the stabilizers of the point $[0:0:1]$ in $\PP^{2}$, we get that 
	$$[\PP^{14}\times \PP^2/\GL_3]\simeq [\PP^{14}/H].$$
	
	The statement follows from noticing that the equation $f([0:0:1])=0$ determines a hyperplane in $\PP^{14}$.
\end{proof}

\begin{remark}\label{rem:H}
	The group $H$ can be described as the subgroup of $\GL_3$ of matrices of the form
	$$
	\begin{pmatrix}
	a & b & 0 \\
	c & d & 0 \\
	f & g & h \\
	\end{pmatrix}.
	$$
	Notice that the submatrix
	$$
	\begin{pmatrix}
	a & b \\
	c & d
	\end{pmatrix}
	$$
	is in $\GL_2$, and in fact we can describe $H$ as $(\GL_2\times \gm)\ltimes \ga^2$.
\end{remark}

This description essentially centers our focus in the point $[0:0:1]$. This implies that the coordinates of the space $\PP^{14}$ can be identified with the coefficient of the Taylor expansion of $f$ in $[0:0:1]$ and this is very useful for computations. \Cref{lem:A-sing} implies that $X$ can be described inside $[\PP^{14}/H]$ as a projective subbundle of codimension $6$. Namely, $X$ is the complete intersection described by the equations $$a_{00}=a_{10}=a_{01}=a_{20}=a_{11}=a_{02}=0$$
where $a_{ij}$ is the coefficient of the monomial $x^iy^jz^{4-i-j}$. One can verify easily that this set of equations is invariant for the action of $H$. This gives us an easy way of computing the fundamental class of $X$ in $\ch_{\GL_3}(\PP^{14}\times\PP^2) \simeq \ch_H(\PP^{14})$.

The fact that $X$ is a projective subbundle implies that it is easy to compute the ideal generated by $\im{\rho_*}$, where $\rho: X\rightarrow \PP^{14}$ is the $\GL_3$-equivariant morphism previously defined. We can repeat the exact same strategy adopted in Section 5 of \cite{FulVis} to prove our version of Theorem 5.5 as in \cite{FulVis}. For the sake of notation, we denote by $h_2$ (respectively $h_{14}$) the hyperplane section of $\PP^2$ (respectively $\PP^{14}$).

\begin{proposition}
	There exists a unique polynomial $p(h_{14},h_{2})$ with coefficients in $\ch(\cB \GL_3)$ such that 
	\begin{itemize}
		\item $p(h_{14},h_{2})$ represents the fundamental class of $X$ in the $\GL_3$-equivariant Chow ring of $\PP^{14}\times \PP^2$,
		\item the degree of $p$ with respect to the variable $h_{14}$ is strictly less than $15$,
		\item the degree of $p$ with respect to the variable $h_{2}$ is strictly less than $3$.
	\end{itemize} 
	Furthermore, if $p$ is of the form $$p_2(h_{14})h_2^2+p_1(h_{14})h_2+p_0(h_{14})$$
	 with $p_0,p_1,p_2 \in \ch_{\GL_3}(\PP^{14})$ and $\deg p_i \leq 14$, we have that $\im{\rho_*}$ is equal to the ideal generated by $p_0$, $p_1$ and $p_2$ in $\ch_{\GL_3}(\PP^{14})$. 
\end{proposition}

\begin{proof}
	For a detailed proof of the proposition see Section 5 of \cite{FulVis}.
\end{proof}

An easy computation of the fundamental class of $X$ in the $\GL_3$-equivariant Chow ring of $\PP^{14}\times \PP^2$ gives us the description of the Chow ring of $\Mtilde_3\setminus (\Detilde_1 \cup \Htilde_3)$.

\begin{corollary}\label{cor:chow-quart}
	We have an isomorphism of rings
	$$ \ch\Big(\Mtilde_3\setminus (\Detilde_1 \cup \Htilde_3)\Big) \simeq \ZZ[\lambda_1,\lambda_2,\lambda_3]/(z_2,p_0,p_1,p_2)$$
	where the generators of the ideal can be described as follows:
	\begin{itemize}
		\item $p_2=12\lambda_1^4 - 44\lambda_1^2\lambda_2 + 92\lambda_1\lambda_3$,
		\item $p_1=-14\lambda_1^3\lambda_2 + 2\lambda_1^2\lambda_3 + 48\lambda_1\lambda_2^2 - 96\lambda_2\lambda_3$,
		\item $p_0=15\lambda_1^3\lambda_3 - 52\lambda_1\lambda_2\lambda_3 + 112\lambda_3^2$,
		\item $z_2=-1152\lambda_1^3\lambda_3^2 + 256\lambda_1^2\lambda_2^2\lambda_3 + 5824\lambda_1\lambda_2\lambda_3^2 - 1152\lambda_2^3\lambda_3 - 10976\lambda_3^3.$
	\end{itemize}
\end{corollary}

\begin{remark}
	We remark that $z_2$ is not in the ideal generated by the other relations, meaning that it was really necessary to introduce the additional stratum of non-reduced curves for the computations.
\end{remark}

\section{Description of $\Detilde_1 \setminus \Detilde_{1,1}$}\label{sec:4}

To describe $\Detilde_1\setminus \Detilde_{1,1}$, we recall the gluing morphism in the case of stable curves:
$$ \Mbar_{1,1} \times \Mbar_{2,1} \longrightarrow \Debar_1,$$ 
which is an isomorphism outside $\Debar_{1,1}$. The same morphism can be defined for the $A_r$-stable case, namely 
$$ \Mtilde_{1,1} \times \Mtilde_{2,1} \longrightarrow \Detilde_1$$ 
and it is an isomorphism outside $\Detilde_{1,1}$.w

Thus, we need to describe the preimage of $\Detilde_{1,1}$ through the morphism. Denote by $\ThTilde_1\subset \Mtilde_{2,1}$ the pullback of $\Detilde_1 \subset \Mtilde_2$ through the morphism $\Mtilde_{2,1} \rightarrow \Mtilde_2$ which forgets the section. Thus one can easily prove that  the preimage of $\Detilde_{1,1}$ through the map 
$$ \Mtilde_{1,1} \times \Mtilde_{2,1} \longrightarrow \Detilde_1$$ 
is equal to $\Mtilde_{1,1} \times \Theta_1$ and therefore
$$ \Mtilde_{1,1}\times (\Mtilde_{2,1}\setminus \ThTilde_1) \simeq \Detilde_1 \setminus \Detilde_{1,1}.$$

To compute the Chow ring of $\Detilde_1\setminus \Detilde_{1,1}$, it remains to describe the stack $\Mtilde_{2,1}\setminus \ThTilde_1$. We start by describing $\Ctilde_2 \setminus \ThTilde_1$, the universal curve over $\Mtilde_2\setminus \Detilde_1$. 

\begin{proposition}
	We have the following isomorphism
	$$  \Ctilde_2 \setminus \ThTilde_1 \simeq [\widetilde{\AA}(6)\setminus 0/B_2],$$ 
	where $B_2$ is the Borel subgroup of lower triangular matrices inside $\GL_2$ and $\widetilde{\AA}(6)$ is a $7$-dimensional $B_2$-representation.
\end{proposition} 

\begin{proof}
	This is a straightforward generalization of Proposition 3.1 of \cite{DiLorPerVis}.We remove the $0$-section because we cannot allow non-reduced curve to appear (condition (b1) in \Cref{def:hyp-A_r}). 
\end{proof}

 The representation $\widetilde{\AA}(6)$ can be described as follows: consider the $\GL_2$-representation $\AA(6)$ of binary forms with coordinates $(x_0,x_1)$ of degree $6$ with an action described by the formula
$$ A.f(x_0,x_1):=\det(A)^{2}f(A^{-1}x),$$  
and consider also the $1$-dimensional $B_2$-representation $\AA^1$ with an action described by the formula 
$$ A.s:=\det(A)a_{22}^{-3}s$$
where
$$
A:=\begin{pmatrix}
	a_{11} & a_{12} \\
	0 & a_{22}
\end{pmatrix}.
$$

We define $\widetilde{\AA}(6)$ to be the invariant subscheme  of $\AA(6)\times \AA^1$ defined by the pairs $(f,s)$ that satisfy the equation $f(0,1)=s^2$. If we use $s$ instead of the coefficient of $x_1^6$ in $f$ as a coordinate, it is clear that $\widetilde{\AA}(6)$ is isomorphic to a $7$-dimensional $B_2$-representation.

The previous proposition gives us the following description of the Chow ring of $\Mtilde_{2,1}\setminus \ThTilde_1$.

\begin{corollary}\label{cor:mtilde_21}
	$\Mtilde_{2,1}\setminus \ThTilde_1$ is the quotient by $B_2$ of the complement of the subrepresentation of $\widetilde{\AA}(6)$ described by the equations $s=a_5=a_4=a_3=0$ where $a_i$ is the coefficient of $x_0^{6-i}x_1^i$.
\end{corollary}

\begin{proof}
	Theorem 2.5 of \cite{Per1} gives that $\Mtilde_{2,1}\setminus \ThTilde_1$ is isomorphic to the open of $\Ctilde_2 \setminus \ThTilde_1$ parametrizing pairs $(C/S,p)$ such that $C$ is a $A_5$-stable curve of genus $2$ and $p$ is a section whose geometric fibers over $S$ are $A_n$-singularity for $n\leq 2$. In particular, we need to describe the closed invariant subscheme $D_3$ of $\widetilde{\AA}(6)$ that parametrizes pairs $(C,p)$ such that $p$ is a $A_n$-singularities with $n\geq 3$. 
	
	To do so, we need to explicit the isomorphism 
	$$ \Ctilde_2 \setminus \ThTilde_1 \simeq [\widetilde{\AA}(6)\setminus 0/B_2]$$
	as described in Proposition 3.1 of \cite{DiLorPerVis}. Given a pair $(f,s) \in \widetilde{\AA}(6)\setminus 0$, we construct a curve $C$ as $\spec_{\PP^1}(\cA)$ where $\cA$ is the finite locally free algebra of rank two over $\PP^1$ defined as $\cO_{\PP^1}\oplus \cO_{\PP^1}(-3)$. The multiplication is induced by the section $f:\cO_{\PP^1}(-6)\into \cO_{\PP^1}$. The section $p$ of $C$ can be defined seeing $s$ as a section of  $\H^0(\cO_{\PP^1}(-3)\vert_{\infty})$, where $\infty \in \PP^1$. This implies that the point $p$ is a $A_n$-singularity for $n\geq 3$ if and only if $\infty$ is a root of multiplicity at least $4$ for the section $f$.  The statement follows.  
\end{proof}
\Cref{cor:mtilde_21} gives us the description of the Chow ring of $\Detilde_{1}\setminus \Detilde_{1,1}$. We denote by $t$ the generator of the Chow ring of $\Mtilde_{1,1}$ defined as $t:=c_1(p^*\cO(p))$ for every object $(C,p)\in \Mtilde_{1,1}$. This is the $\psi$-class of $\Mtilde_{1,1}$.

\begin{proposition}\label{prop:descr-detilde-1}
	We have the following isomorphism
	$$ \ch(\Detilde_1 \setminus \Detilde_{1,1}) \simeq \ZZ[1/6,t_0,t_1,t]/(f)$$
	where $f \in \ZZ[1/6,t_0,t_1]$ is the polynomial:
	$$f=2t_1(t_0+t_1)(t_0-2t_1)(t_0-3t_1).$$
\end{proposition}

\begin{proof}
	Because $\Mtilde_{1,1}$ is a vector bundle over $\cB \gm$, the morphism 
	$$ \ch(\Mtilde_{2,1}\setminus \ThTilde_1)\otimes \ch(\cB\gm) \rightarrow \ch(\Detilde_1 \setminus \Detilde_{1,1})$$
	is an isomorphism. Therefore it is enough to describe the Chow ring of $\Mtilde_{2,1}\setminus \ThTilde_1$. The previous corollary gives us that if $T_2$ is the maximal torus of $B_2$ and $t_0,t_1$ are the two generators for the character group of $T_2$, we have that 
	$$\ch(\Mtilde_{2,1}\setminus \ThTilde_1)\simeq \ZZ[1/6,t_0,t_1]/(f)$$
	where $f$ is the fundamental class associated with the vanishing of the coordinates $s,a_5,a_4,a_3$ of $\widetilde{\AA}(6)$. Again, we are using Proposition 2.3 of \cite{MolVis}.  This computation in the $\gm^2$-equivariant setting is straightforward and it gives us the result.
\end{proof}

Now, we use the results in Appendix A of \cite{DiLorVis} to describe the first Chern of the normal bundle of $\Detilde_1 \setminus \Detilde_{1,1}$ in $\Mtilde_3 \setminus \Detilde_{1,1}$.

\begin{proposition}\label{prop:relation-detilde-1}
	The closed immersion $\Detilde_1\setminus \Detilde_{1,1} \into \Mtilde_3 \setminus \Detilde_{1,1}$ is regular and the first Chern class of the normal bundle is of equal to $t+t_1$. Moreover, we have the following equalities in the Chow ring of $\Detilde_{1}\setminus \Detilde_{1,1}$:
	\begin{itemize}
		\item[(1)] $\lambda_1=-t-t_0-t_1$,
		\item[(2)] $\lambda_2=t_0t_1+t(t_0+t_1)$,
		\item[(3)] $\lambda_3=-t_0t_1t$,
		\item[(4)] $[H]\vert_{\Detilde_{1}\setminus \Detilde_{1,1}}=t_0-2t_1$.
	\end{itemize} 
\end{proposition}

\begin{proof}
The closed immersion is regular because the two stacks are smooth.  Because of Appendix A of \cite{DiLorVis}, we know that the normal bundle is the determinant of a vector bundle $N$ of rank $2$ tensored by a $2$-torsion line bundle whose first Chern class vanishes when we invert $2$ in the Chow ring. Moreover, we can describe $N$ in the following way. Suppose $(C/S,p)$ is an object of $(\Detilde_{1}\setminus \Detilde_{1,1})(S)$, where $p$ is the section whose geometric fibers over $S$ are separating nodes. Then $N$ is the normal bundle of $p$ inside $C$, which has rank $2$. If we decompose $C$ as a union of an elliptic curve $(E,e)\in \Mtilde_{1,1}$ and a $1$-pointed genus $2$ curve $(C_0,p_0) \in \Mtilde_{2,1}\setminus ThTilde_1$, it is clear that $N=N_{e/E}\oplus N_{p_0/C_0}$. Therefore it is enough to compute the two line bundles separately.
	
  The first Chern class of the line bundle $N_{e/E}$ is exactly the $\psi$-class of $\Mtilde_{1,1}$, which is $t$ by definition. Similarly, the line bundle $N_{p_0/C_0}$ is the $\psi$-class of $\Mtilde_{2,1}$. Using the description of $\Ctilde_2$ as a quotient stack (see the proof of \Cref{cor:mtilde_21}), one can prove that 
	  \begin{itemize}
	  	\item $N_{p_0/C_0}=\cO_{\PP^1}(1)^{\vee}\vert_{p_{\infty}}=E_1$, where $E_1$ is the character of $B_2$ whose Chern class is $t_1$;
	  	\item $\pi_*\omega_{C_0}=E_0^{\vee}\oplus E_1^{\vee}$, where $E_0$ is the character of $B_2$ whose Chern class is $t_0$.
	\end{itemize}

We now prove the description of the fundamental class of the hyperelliptic locus. We have that $\Htilde_3$ intersects transversely $\Detilde_1\setminus \Detilde_{1,1}$ in the locus where the section of $\Mtilde_{2,1}$ is a Weierstrass point for the (unique) involution of a genus $2$ curve. The description of $\Ctilde_2 \setminus \ThTilde_1$ as a quotient stack (see the proof of \Cref{cor:mtilde_21}) implies that $[H]\vert_{\Detilde_1\setminus \Detilde_{1,1}}$ is equal to the class of the vanishing locus $s=0$ in $\widetilde{\AA}(6)$, which is easily computed as we know explicitly the action of $B_2$. The only thing that remains to do is the descriptions of the $\lambda$-classes. They follow from the following lemma.
\end{proof}

\begin{lemma}\label{lem:hodge-sep}
	Let $C/S$ be an $A_r$-stable curve over a base scheme $S$ and $p$ a separating node (i.e. a section whose geometric fibers over $S$ are separating nodes). Consider $b:\widetilde{C}\arr C$ the blowup of $C$ in $p$ and denote by $\widetilde{\pi}$ the composition $\pi \circ b$. Then we have $$\pi_*\omega_{C/S}\simeq \widetilde{\pi}_*\omega_{\widetilde{C}/S}.$$
\end{lemma}

\begin{proof}
	Denote by $D$ the dual of the conductor ideal. We know by the Noether formula that $b^*\omega_{C/S}=\omega_{\widetilde{C}/S}(D)$, therefore if we tensor the exact sequence 
	$$0 \rightarrow \cO_C \rightarrow f_*\cO_{\widetilde{C}} \rightarrow Q \rightarrow 0$$ 
	by $\omega_{C/S}$, we get the following injective morphism 
	$$ \omega_{C/S} \into f_*\omega_{\widetilde{C}/S}(D). $$
	As usual, the smoothness of $\Mtilde_g^r$ implies that we can apply Grauert's theorem to prove that the line bundles $\omega_{\widetilde{C}}$ and $\omega_{\widetilde{C}}(D)$ satisfy base change over $S$. Consider now the morphism on global sections
	$$ \pi_*\omega_{C/S} \into\widetilde{\pi}_*\omega_{\widetilde{C}/S}(D),$$
	because they both satisfy base change, we can prove the surjectivity restricting to the geometric points of $S$. The statement for algebraically closed fields has already been proved in Lemma 3.38 of \cite{Per2}. 
	
	Finally, given the morphism 
	$$ \omega_{\widetilde{C}} \into \omega_{\widetilde{C}}(D) $$ 
	we consider the global sections 
	$$ \widetilde{\pi}_*\omega_{\widetilde{C}/S} \into \widetilde{\pi}_*\omega_{\widetilde{C}/S}(D)$$ 
	and again the surjectivity follows restricting to the geometric fibers over $S$. 
\end{proof}

\section{Description of $\Detilde_{1,1} \setminus \Detilde_{1,1,1}$}\label{sec:5}

Let us consider the following morphism 
$$ \Mtilde_{1,2} \times \Mtilde_{1,1} \times \Mtilde_{1,1} \longrightarrow \Detilde_{1,1}$$ 
defined by the association $$(C,p_1,p_2),(E_1,e_1),(E_2,e_2) \mapsto E_1 \cup_{e_1\equiv p_1} C \cup_{p_2 \equiv e_2} E_2$$
where the cup symbol represent the gluing of the two curves using the two points specified. The operation is obviously non commutative.

The preimage of $\Detilde_{1,1,1}$ through the morphism is the product $\Mtilde_{1,1}\times \Mtilde_{1,1}\times \Mtilde_{1,1}$ where $\Mtilde_{1,1}\subset \Mtilde_{1,2}$ is the universal section of the natural functor $\Mtilde_{1,2}\rightarrow \Mtilde_{1,1}$.

\begin{lemma}\label{lem:detilde-1-1}
	The morphism 
	$$ \pi_2:(\Mtilde_{1,2}\setminus \Mtilde_{1,1}) \times \Mtilde_{1,1} \times \Mtilde_{1,1} \longrightarrow \Detilde_{1,1}\setminus \Detilde_{1,1,1}$$ 
	described above is a $C_2$-torsor.
\end{lemma}

\begin{proof}
	The stack $\Detilde_{1,1} \setminus \Detilde_{1,1,1}$ parametrizes pairs  $(C,E)$ where $C$ is a $A_7$-stable curve of genus $3$ and $E$ is a finite \'etale divisor of degree $2$ of $C$ whose support the disjoint union of two separating nodes. Proposition 3.13 and Proposition 3.16 of \cite{Per1} gives us that $(C,E)$ is completely determined by the partial normalization $b:\widetilde{C} \rightarrow C$ of $C$ in $E$ and by the fiber $b^{-1}(E)$. It follows easily that the action of $C_2$ on $(\Mtilde_{1,2}\setminus \Mtilde_{1,1}) \times \Mtilde_{1,1} \times \Mtilde_{1,1}$ induced by the involution 
	$$ \Big((C,p_1,p_2),(E_1,e_1),(E_2,e_2)\Big) \mapsto \Big((C,p_2,p_1),(E_2,e_2),(E_1,e_1)\Big)$$ 
    induces a structure of $C_2$-torsor on the morphism $\pi_2$.
\end{proof}

Denote by $\Ctilde_{1,1}$ the universal curve over $\Mtilde_{1,1}$. It naturally comes with a universal section $\Mtilde_{1,1}\into \Ctilde_{1,1}$.

\begin{lemma}\label{lem:mtilde-12}
	We have the isomorphism 
	$$ \Mtilde_{1,2}\setminus\Mtilde_{1,1} \simeq \Ctilde_{1,1}\setminus \Mtilde_{1,1}$$
	and therefore we have the following isomorphism of rings
	$$ \ch( \Mtilde_{1,2}\setminus\Mtilde_{1,1})\simeq \ZZ[1/6,t]. $$  
\end{lemma}

\begin{proof}
 The isomorphism is a corollary of Theorem 2.5 of \cite{Per1}. The computation of the Chow ring of $\Ctilde_{1,1}\setminus \Mtilde_{1,1}$  is Lemma 3.2 of \cite{DiLorPerVis}.
\end{proof}

\begin{remark}
	It is important to notice that $\Mtilde_{1,2}$ has at most $A_3$-singularities while $\Ctilde_{1,1}$ has at most cusps, because it is the universal curve of $\Mtilde_{1,1}$.
\end{remark}
\begin{corollary}
	The algebraic stack $\Detilde_{1,1}\setminus \Detilde_{1,1,1}$ is smooth.
\end{corollary}

 \Cref{lem:chow-tor} implies that 
$$\ch(\Detilde_{1,1}\setminus \Detilde_{1,1,1})\simeq \ZZ[1/6,t,t_1,t_2]^{\rm inv},$$
where the action of $C_2$ is defined on object by the following association 
$$ \Big( (C,p_1,p_2), (E_1,e_1), (E_2,e_2) \Big) \mapsto \Big((C,p_2,p_1), (E_2,e_2), (E_1,e_1)\Big).$$

A computation shows that the involution acts on the Chow ring of the product in the following way
$$ (t,t_1,t_2) \mapsto (t,t_2,t_1)$$
and therefore we have the description we need.

\begin{proposition}
	In the situation above, we have the following isomorphism:
	$$ \ch(\Detilde_{1,1}\setminus \Detilde_{1,1,1})\simeq \ZZ[1/6,t,c_1,c_2]$$ 
	where $c_1:=t_1+t_2$ and $c_2:=t_1t_2$.
\end{proposition}

It remains to describe the normal bundle of the closed immersion $\Detilde_{1,1}\setminus \Detilde_{1,1,1} \into \Mtilde_{3}\setminus\Detilde_{1,1,1}$ and the other classes. As usual, we denote by $\delta_1$ the fundamental class of $\Detilde_1$ in $\Mtilde_3$.

\begin{proposition}\label{prop:relat-detilde-1-1}
	We have the following equalities in the Chow ring of $\Detilde_{1,1}\setminus \Detilde_{1,1,1}$:
	\begin{itemize}
		\item[(1)] $\lambda_1=-t-c_1$,
		\item[(2)] $\lambda_2=c_2+tc_1$,
		\item[(3)] $\lambda_3=-tc_2$,
		\item[(4)] $[H]\vert_{\Detilde_{1,1}\setminus \Detilde_{1,1,1}} = -3t$,
		\item[(5)] $\delta_1\vert_{\Detilde_{1,1}\setminus \Detilde_{1,1,1}} = 2t+c_1$.
	\end{itemize}
	Furthermore, the second Chern class of the normal bundle of the closed immersion $\Detilde_{1,1}\setminus \Detilde_{1,1,1} \into \Mtilde_{3}\setminus\Detilde_{1,1,1}$ is equal to $c_2+tc_1+t^2$.
\end{proposition} 

\begin{proof}
	The restriction of the $\lambda$-classes can be computed using \Cref{lem:hodge-sep}. The proof of formula (5) is exactly the same of the computation of the normal bundle of $\Detilde_{1}\setminus \Detilde_{1,1}$ in $\Mtilde_3 \setminus \Detilde_{1,1}$. The only thing to remark is that the two $\psi$-classes of $\Mtilde_{1,2}\setminus \Mtilde_{1,1}$ coincide with the generator $t$ of $\ch(\Mtilde_{1,2}\setminus \Mtilde_{1,1})$.
	
	As far as the fundamental class of the hyperelliptic locus is concerned, it is clear that it coincides with the fundamental class of the locus in $\Mtilde_{1,2}\setminus \Mtilde_{1,1}$ parametrizing $2$-pointed stable curves of genus $1$ such that the two sections are both fixed points for an involution. This computation can be done using the description of $\Mtilde_{1,2}\setminus \Mtilde_{1,1}$ as $\Ctilde_{1,1}\setminus \Mtilde_{1,1}$, in particular as in Lemma 3.2 of \cite{DiLorPerVis}. In fact, they proved that $\Ctilde_{1,1}\setminus \Mtilde_{1,1}$ is an invariant subscheme $W$ of a $\gm$-representation $V$ of dimension $4$, where the action can be described as $$t.(x,y,\alpha,\beta)=(t^{-2}x,t^{-3}y,t^{-4}\alpha,t^{-6}\beta)$$
	for every $t \in \gm $ and $(x.y.\alpha,\beta) \in V$. Specifically, $W$ is the hypersurface in $V$ defined by the equation $y^2=x^3+\alpha x+\beta$, which is exactly the dehomogenization of the Weierstrass form of an elliptic curve with a flex at the infinity. A straightforward computation shows that the hyperelliptic locus is defined by the equation $y=0$. 
	
	Finally, the normal bundle of the closed immersion can be described using the theory developed in Appendix A of \cite{DiLorVis} as the sum 
	$$(N_{p_1|C}\otimes N_{e_1|E_1})\oplus (N_{p_2|C}\otimes N_{e_2|E_2})$$
	for every element $[(C,p_1,p_2), (E_1,e_1), (E_2,e_2)]$ in $\Detilde_{1,1}\setminus \Detilde_{1,1,1}$. A straightforward computation using Chern classes implies the result.
\end{proof}

\section{Description of $\Detilde_{1,1,1}$}\label{sec:6}

Finally, to describe the last strata, we use the morphism introduced in the previous paragraph to define
$$c_6:\Mtilde_{1,1} \times \Mtilde_{1,1}\times \Mtilde_{1,1} \rightarrow \Detilde_{1,1,1}$$
which can be described as taking three elliptic (possibly singular) curves $(E_1,e_1)$, $(E_2,e_2)$, $(E_3,e_3)$ and attach them to a projective line with three distinct points $(\PP^1,0,1,\infty)$ using the order $e_1\equiv 0$, $e_2\equiv 1$, $e_3 \equiv \infty$. We denote by $S_3$ the group of permutation of a set of three elements.

\begin{lemma}\label{lem:descr-delta-1-1-1}
	The morphism $c_6$ is a $S_3$-torsor.
\end{lemma}

\begin{proof}
	The proof of \Cref{lem:detilde-1-1} can be adapted perfectly to this statement. 
\end{proof}

The previous lemma implies that we have an action of $S_3$ on $\ch(\Mtilde_{1,1}^{\times 3})$. Therefore it is clear that 
$$ \ch(\Detilde_{1,1,1})=\ZZ[1/6,c_1,c_2,c_3]$$
where $c_i$ is the $i$-th symmetric polynomial in the variables $t_1,t_2,t_3$, which are the generators of the Chow rings of the factors.. 

\begin{proposition}
	We have that the following equalities in the Chow ring of $\Detilde_{1,1,1}$:
	\begin{itemize}
		\item[(1)] $\lambda_1=-c_1$,
		\item[(2)] $\lambda_2=c_2$,
		\item[(3)] $\lambda_3=-c_3$,
		\item[(4)] $[H]\vert_{\Detilde_{1,1,1}}=0$,
		\item[(5)] $\delta_1\vert_{\Detilde_{1,1,1}}=c_1$,
		\item[(6)] $\delta_{1,1}\vert_{\Detilde_{1,1,1}}=c_2$.
	\end{itemize}
	Furthermore, the third Chern class of the normal bundle of the closed immersion $\Detilde_{1,1,1}\into \Mtilde_3$ is equal to $c_3$.
\end{proposition}

\begin{proof}
	We can use \Cref{lem:hodge-sep} to get the description of  the $\lambda$-classes. To compute $\delta_1$ and $\delta_{1,1}$ and the third Chern class of the normal bundle one can use again the results of Appendix A of \cite{DiLorVis} and adapt the proof of \Cref{prop:relat-detilde-1-1}. Notice that in this case $$N_{\Detilde_{1,1,1}|\Mtilde_3}=N_{e_1|E_1}\oplus N_{e_2|E_2}\oplus N_{e_3|E_3}.$$ Finally, we know that there are no hyperelliptic curves of genus $3$ with three separating nodes. In fact, any involution restricts to the identity over the projective line, because it fixes three points. But then its fixed locus is not finite, therefore it is not hyperelliptic. 
\end{proof}

\section{The gluing procedure and the Chow ring of $\Mtilde_3$}\label{sec:7}

In the last section of this chapter, we explain how to calculate explicitly the Chow ring of $\Mtilde_3$. It is not clear a priori how to describe the fiber product that appears in \Cref{lem:gluing}.

Let $\cU$, $\cZ$ and $\cX$ as in \Cref{lem:gluing} and denote by $i:\cZ \into \cX$ the closed immersion and by $j:\cU \into \cZ$ the open immersion.  Let us set some notations.
\begin{itemize}
\item[$\bullet$] $ \ch(\cU)$ is generated by the elements $x'_1,\dots,x'_r$ and let $x_1,\dots,x_r$ be some liftings of $x_1',\dots,x_r'$ in $\ch(\cX)$; we denote by $\eta$ the morphism 
$$ \ZZ[1/6,X_1,\dots,X_r,Z] \longrightarrow \ch(\cX)$$
where $X_h$ maps to $x_h$ for $h=1,\dots,r$ and $Z$ maps to $[\cZ]$, the fundamental class of $\cZ$; we denote by $\eta_{\cU}$ the composite 
$$ \ZZ[1/6,X_1,\dots,X_r] \into \ZZ[1/6,X_1,\dots,X_r,Z] \longrightarrow \ch(\cX)\rightarrow \ch(\cU).$$ 
Furthermore, we denote by $p_1'(X), \dots, p_n'(X)$ a choice of generators for $\ker(\eta_{\cU})$, where $X:=(X_1,\dots,X_r)$.
\item[$\bullet$]  $ \ch(\cZ)$ is generated by elements $y_1,\dots,y_s \in \ch(\cZ)$; we denote by $a$ the morphism 
$$ \ZZ[1/6,Y_1,\dots,Y_s] \longrightarrow \ch(\cZ)$$ 
where $Y_h$ maps to $y_h$ for $h=1,\dots,s$. Furthermore, we denote by $q_1'(Y), \dots, q_m'(Y)$ a choice of generators for $\ker(a)$, where $Y:=(Y_1,\dots,Y_s)$.
\end{itemize}

Because $a$ is surjective, there exists a morphism
$$ \eta_{\cZ}: \ZZ[1/6,X_1,\dots,X_r,Z] \longrightarrow \ZZ[1/6,Y_1,\dots,Y_r]$$
which is a lifting of the morphism $i^*$, i.e. it makes the following diagram
$$
\begin{tikzcd}
	{\ZZ[1/6,X_1,\dots,X_r,Z]} \arrow[d, "\eta"] \arrow[r, "\eta_{\cZ}", dotted] & {\ZZ[1/6,Y_1,\dots,Y_s]} \arrow[d, "a"] \\
	\ch(\cX) \arrow[r, "i^*"]                                                    & \ch(\cZ)                               
\end{tikzcd}
$$
commute.

The cartesianity of the diagram in \Cref{lem:gluing} implies the following lemma.

\begin{lemma}\label{lem:gluing-sur}
	In the situation above, suppose that the morphism $\eta_{\cZ}$ is surjective. Then $\eta$ is surjective.
\end{lemma}

\begin{proof}
	This follows because $ j^*\circ \eta$ is always surjective and the hypothesis implies that also $i^*\circ \eta$ is surjective.
\end{proof}

In the hypothesis of the lemma, we can also describe explicitly the relations, i.e. the kernel of $\eta$. First of all, denote by $q_h(X,Z)$ some liftings of $q_h'(Y)$ through the morphism $\eta_{\cZ}$ for $h=1,\dots,m$. A straightforward computation shows that we have $Zq_h(X,Z) \in \ker \eta$ for $h=1,\dots,m$.
We have found our first relations. We refer to these type of relations as  \emph{liftings of the closed stratum's relations}. 

Another important set of relations comes from the kernel of $\eta_{\cZ}$. In fact, if $v \in \ker\eta_{\cZ}$, then a simple computation shows that $\eta(Zv)=0$. Therefore if $v_1,\dots,v_l$ are generators of $\ker\eta_{\cZ}$, we get that $Zv_h \in \ker\eta$ for $h=1,\dots,l$.

Finally, the last set of relations are the relations that come from $\cU$, the open stratum. The element $p_h(X)$ in general is not a relation as its restriction to $\ch(\cZ)$ may fail to vanish. We can however do the following procedure to find a modification of $p_h$ which still vanishes restricted to the open stratum and it is in the kernel of $\eta$. Recall that we have a well-defined morphism
$$ i^*:\ch(\cU) \longrightarrow \ch(\cZ)/(c_{\rm top}(N_{\cZ|\cX}))$$ 
which implies that $\eta_{\cZ}(p_h) \in (q'_1,\dots,q'_m,c_{\rm top}(N_{\cZ|\cX}))$. We choose an element $g'_h \in \ZZ[1/6,Y_1,\dots,Y_s]$ such that 
$$ \eta_{\cZ}(p_h) + g_h'c_{\rm top}(N_{\cZ|\cX}) \in (q'_1,\dots,q'_m)$$
and consider a lifting $g_h$ of $g'_h$ through the morphism $\eta_{\cZ}$. A straightforward computation shows that $p_h+Zg_h \in \ker \eta$ for every $h=1,\dots,n$.

\begin{proposition}\label{prop:desc-gluing}
	In the situation above, $\ker\eta$ is generated by the elements $Zq_1,\dots,Zq_m,Zv_1,\dots,Zv_l,p_1+Zg_1,\dots,p_n+Zg_n$.
\end{proposition}

\begin{proof}
	Consider the following commutative diagram
	$$
	\begin{tikzcd}
		{\ZZ[1/6,X_1,\dots,X_r,Z]} \arrow[r, "\eta"] \arrow[d, "\eta_{\cZ}"] & \ch(\cX) \arrow[d, "i^*"] \arrow[r, "j^*"] & \ch(\cU) \arrow[d, "i^*"]           \\
		{\ZZ[1/6,Y_1,\dots,Y_s]} \arrow[r, "a"]                              & \ch(\cZ) \arrow[r, "b"]                    & \ch(\cZ)/(c_{\rm top}(N_{\cZ|\cX}))
	\end{tikzcd}
	$$
	where $b$ is the quotient morphisms. Recall that the right square is cartesian. Notice that because $\eta_{\cZ}$ is surjective, all the other morphisms are surjectives.
	
	We denote by $c$ the top Chern class of the normal bundle $N_{\cZ|\cX}$. Let $p(X)+Zg(X,Z)$ be an element in $\ker \eta$. Because we have $$(j^*\circ \eta)(p(X)+Zg(X,Z))=0,$$ thus $p \in (p_1',\dots,p_n')$ which implies there exists $b_1,\dots,b_n \in \ZZ[1/6,X_1,\dots,X_r]$ such that 
	$$ p = \sum_{h=1}^n b_hp_h . $$
	Now, we pullback the element to $\cZ$, thus we get
	$$ (i^*\circ \eta)(p+Zg)=a\Big(\sum_{h=1}^n \eta_{\cZ}(b_h) \eta_{\cZ}(p_h) + c \eta_{\cZ}(g)\Big)$$ 
	or equivalently   
	$$\sum_{h=1}^n \eta_{\cZ}(b_h) \eta_{\cZ}(p_h) + c \eta_{\cZ}(g) \in (q_1',\dots, q_m').$$
	By construction $\eta_{\cZ}(p_h)=-g_h'c + (q_1',\dots,q_m')$ and $\eta_{\cZ}(g_h)=g_h'$, therefore we get 
	$$ c  \eta_{\cZ}\Big(g-\sum_{h=1}^nb_hg_h\Big) \in (q_1',\dots,q_m'). $$
	Because $c$ is a non-zero divisor in $\ch(\cZ)$ we have that
	$$ \eta_{\cZ}\Big(g-\sum_{h=1}^nb_hg_h\Big) \in (q_1',\dots,q_m')$$ 
	or equivalently 
	$$g= \sum_{h=1}^n b_hg_h + t$$
	with $t \in \ker(a \circ \eta_{\cZ})$. Therefore we have that
	$$ p+Zg=\sum_{h=1}^n b_h(p_h+Zg_h)+Zt$$
	with $t \in \ker(a \circ \eta_{\cZ})$. One can check easily that $\ker(a \circ \eta_{\cZ})$ is generated by $(v_1,\dots,v_l,q_1,\dots,q_m)$.
\end{proof}

Now, we sketch how to apply this procedure to the first two strata, namely $\Htilde_3\setminus \Detilde_{1}$ and $\Mtilde_3 \setminus (\Htilde_3 \cup \Detilde_1)$, to get the Chow ring of $\Mtilde_3 \setminus \Detilde_1$. This is the most complicated part as far as computations is concerned. The other gluing procedures are left to the curious reader as they follows the exact same ideas. 

In our situation, we have $\cU:=\Mtilde_{3}\setminus (\Htilde_3 \cup \Detilde_1)$ and $\cZ:= \Htilde_3 \setminus \Detilde_1$. We know the description of their Chow rings thanks to \Cref{cor:chow-hyper} and \Cref{cor:chow-quart}. Let us look at the generators we need. The Chow ring of $\cU$ is generated by $\lambda_1$, $\lambda_2$ and $\lambda_3$. The Chow ring of $\cZ$ is generated by $\lambda_1$, $\lambda_2$ and $\xi_1$. Therefore the morphism 
$$\eta_{\cZ} : \ZZ[1/6,\lambda_1,\lambda_2,\lambda_3, H] \longrightarrow \ZZ[1/6,\lambda_1,\lambda_2,\xi_1]$$ 
is surjective because $\eta_{\cZ}(H)=(2\xi_1-\lambda_1)/3$. We can also describe the $\ker \eta_{\cZ}$ which is generated by any lifting of the description of $\lambda_3$ in $\cZ$ (see \Cref{lem:lambda-class-H}). This gives us our first relation (after multiplying it by $H$, the fundamental class of the hyperelliptic locus). Furthermore, we can consider the ideal of relations in $\Htilde_{3}\setminus \Detilde_1$ which is generated by the relations $c_9$, $D_1$ and $D_2$ we described in \Cref{cor:chow-hyper}. Therefore we have other three relations. 

Lastly, we consider the four relations as in \Cref{cor:chow-quart} and compute their image through $\eta_{\cZ}$. The hardest part is to find a description of these elements in terms of the generators of the ideal of relations of $\Htilde_3 \setminus \Detilde_1$ and the top Chern class of the normal bundle of the closed immersion. We do not go into details about the computations, but the main idea is to notice that every monomial of the polynomials we need to describe can be written in terms of the relations and the top Chern class.

We state our theorem, which gives us the description of the Chow ring of $\Mtilde_3$. We write the explicit relations in \Cref{rem:relations-Mtilde}.

\begin{theorem}\label{theo:main}
	We have the following isomorphism 
	$$ \ch(\Mtilde_3)\simeq \ZZ[1/6,\lambda_1,\lambda_2,\lambda_3,H,\delta_1,\delta_{1,1},\delta_{1,1,1}]/I$$
	where $I$ is generated by the following relations:
	\begin{itemize}
		\item $k_h$, which comes from the generator of $\ker i_H^*$, where $i_H: \Htilde_3\setminus \Detilde_1 \into \Mtilde_3\setminus \Detilde_1$;
		\item $k_{1}(1)$ and $k_1(2)$, which come from the two generators of $\ker i_{1}^*$ where $i_{1}: \Detilde_1\setminus \Detilde_{1,1} \into \Mtilde_3\setminus \Detilde_{1,1}$;
		\item $k_{1,1}(1)$, $k_{1,1}(2)$ and $k_{1,1}(3)$, which come from the three generators of $\ker i_{1,1}^*$ where $i_{1,1}: \Detilde_{1,1}\setminus \Detilde_{1,1,1} \into \Mtilde_3\setminus \Detilde_{1,1,1}$;
		\item $k_{1,1,1}(1)$, $k_{1,1,1}(2)$, $k_{1,1,1}(3)$ and $k_{1,1,1}(4)$, which come from the four generators of $\ker i_{1,1,1}^*$ where $i_{1,1,1}: \Detilde_{1,1,1} \into \Mtilde_3$;
		\item $m(1)$, $m(2)$, $m(3)$ and $r$, which are the litings of the generators of the relations of the open stratum $\Mtilde_3\setminus (\Htilde_3 \cup \Detilde_1)$;
		\item $h(1)$, $h(2)$ and $h(3)$, which are the liftings of the generators of the relations of the stratum $\Htilde_3 \setminus \Detilde_1$; 
		\item $d_1(1)$, which is the lifting of the generator of the relations of the stratum $\Detilde_1\setminus \Detilde_{1,1}$.
	\end{itemize}
	Furthemore, $h(2)$, $h(3)$ and $d_1(1)$ are in the ideal generated by the other relations. 
\end{theorem}

\begin{remark}\label{rem:relations-Mtilde}
	We write explicitly the relations we need to generate the ideal. 
	\begin{itemize}
		\item[]
		\begin{equation*}
		\begin{split}
		k_h& =\frac{1}{8}\lambda_1^3H + \frac{1}{8}\lambda_1^2H^2 + \frac{1}{4}\lambda_1^2H\delta_1 - \frac{1}{2}\lambda_1\lambda_2H - \frac{1}{8}\lambda_1H^3 +
		\frac{7}{8}\lambda_1H\delta_1^2 + \\ & + \frac{3}{2}\lambda_1\delta_1\delta_{1,1} - \frac{1}{2}\lambda_2H^2 + \lambda_3H - \frac{1}{8}H^4 - \frac{1}{4}H^3\delta_1 +
		\frac{1}{8}H^2\delta_1^2 + \frac{3}{4}H\delta_1^3 + \\ & + \frac{3}{2}\delta_1^2\delta_{1,1}
		\end{split}
		\end{equation*}
		
		\item[] 
		\begin{equation*}
		\begin{split}
		k_1(1)=\frac{1}{4}\lambda_1^2\delta_1 + \frac{1}{2}\lambda_1H\delta_1 + 2\lambda_1\delta_1^2 + \lambda_2\delta_1 + \frac{1}{4}H^2\delta_1 + H\delta_1^2 + \frac{7}{4}\delta_1^3 -\delta_1\delta_{1,1}
		\end{split}
		\end{equation*}
		
		\item[] 
		\begin{equation*}
		\begin{split}
		k_1(2)& =\frac{1}{4}\lambda_1^3\delta_1 + \frac{1}{2}\lambda_1^2H\delta_1 + \frac{5}{4}\lambda_1^2\delta_1^2 + \frac{1}{4}\lambda_1H^2\delta_1 + \frac{3}{2}\lambda_1H\delta_1^2 +
		\frac{7}{4}\lambda_1\delta_1^3 + \\ & + \lambda_1\delta_1\delta_{1,1} - \lambda_1\delta_{1,1,1} + \lambda_3\delta_1+  \frac{1}{4}H^2\delta_1^2 + H\delta_1^3 + \frac{3}{4}\delta_1^4 +
		\delta_1^2\delta_{1,1}
		\end{split}
		\end{equation*}
		
		\item[] 
		\begin{equation*}
		\begin{split}
		k_{1,1}(1)=-\lambda_1^2\delta_{1,1} - 2\lambda_1\delta_1\delta_{1,1} - \lambda_2\delta_{1,1} - \delta_1^2\delta_{1,1} + \delta_{1,1}^2
		\end{split}
		\end{equation*}
		
		\item[] 
		\begin{equation*}
		\begin{split}
		k_{1,1}(2)=3\lambda_1\delta_{1,1} + H\delta_{1,1} + 3\delta_1\delta_{1,1}
		\end{split}
		\end{equation*}
		
		\item[] 
		\begin{equation*}
		\begin{split}
		k_{1,1}(3)&=2\lambda_1^3\delta_{1,1} + 5\lambda_1^2\delta_1\delta_{1,1} + \lambda_1\lambda_2\delta_{1,1} + 4\lambda_1\delta_1^2\delta_{1,1} + \lambda_2\delta_1\delta_{1,1} + \\ & + \lambda_2\delta_{1,1,1} + \lambda_3\delta_{1,1} +
		\delta_1^3\delta_{1,1}
		\end{split}
		\end{equation*}
		
		\item[] 
		\begin{equation*}
		\begin{split}
		k_{1,1,1}(1)=\lambda_1\delta_{1,1,1} + \delta_1\delta_{1,1,1}
		\end{split}
		\end{equation*}
		
		\item[] 
		\begin{equation*}
		\begin{split}
		k_{1,1,1}(2)=\lambda_2\delta_{1,1,1} - \delta_{1,1}\delta_{1,1,1}
		\end{split}
		\end{equation*}
		
		\item[] 
		\begin{equation*}
		\begin{split}
		k_{1,1,1}(3)=\lambda_3\delta_{1,1,1} + \delta_{1,1,1}^2
		\end{split}
		\end{equation*}		
	
		\item[] 
		\begin{equation*}
		\begin{split}
		k_{1,1,1}(4)=H\delta_{1,1,1}
		\end{split}
		\end{equation*}
		
		\item[] 
		\begin{equation*}
		\begin{split}
		m(1)&=12\lambda_1^4 - \frac{7}{3}\lambda_1^3H + 27\lambda_1^3\delta_1 - 44\lambda_1^2\lambda_2 - \frac{706}{9}\lambda_1^2H^2 - \frac{65}{2}\lambda_1^2H\delta_1 + \\ & 
		+ 84\lambda_1^2\delta_1^2 - 32\lambda_1^2\delta_{1,1}  - 38\lambda_1\lambda_2H + 92\lambda_1\lambda_3 - \frac{715}{9}\lambda_1H^3 - \\ & -
		\frac{1340}{9}\lambda_1H^2\delta_1 - 25\lambda_1H\delta_1^2 + 69\lambda_1\delta_1^3  - 130\lambda_1\delta_1\delta_{1,1} + 92\lambda_1\delta_{1,1,1} + \\ & +
		6\lambda_2H^2 - \frac{46}{3}H^4 - \frac{1205}{18}H^3\delta_1 - \frac{562}{9}H^2\delta_1^2 - \frac{101}{6}H\delta_1^3 -
		54\delta_1^2\delta_{1,1}
		\end{split}
		\end{equation*}
		
		\item[] 
		\begin{equation*}
		\begin{split}
		m(2)&=-\frac{55}{18}\lambda_1^4H + \frac{9}{2}\lambda_1^4\delta_1 - 14\lambda_1^3\lambda_2 - \frac{31}{3}\lambda_1^3H^2 - \frac{58}{9}\lambda_1^3H\delta_1 +
		69\lambda_1^3\delta_1^2 - \\ & - \frac{272}{3}\lambda_1^3\delta_{1,1} -  \frac{173}{9}\lambda_1^2\lambda_2H + 2\lambda_1^2\lambda_3 - \frac{137}{18}\lambda_1^2H^3
		- \frac{167}{4}\lambda_1^2H^2\delta_1 + \\ & + \frac{1831}{36}\lambda_1^2H\delta_1^2 + \frac{459}{2}\lambda_1^2\delta_1^3 - 
		\frac{461}{3}\lambda_1^2\delta_1\delta_{1,1} + 2\lambda_1^2\delta_{1,1,1} + 48\lambda_1\lambda_2^2 + \\ & + \frac{1}{9}\lambda_1\lambda_2H^2 - \frac{1}{3}\lambda_1H^4 -
		\frac{605}{18}\lambda_1H^3\delta_1 - \frac{955}{18}\lambda_1H^2\delta_1^2 +  139\lambda_1H\delta_1^3 + \\ & + 291\lambda_1\delta_1^4 -
		49\lambda_1\delta_1^2\delta_{1,1} + 48\lambda_2^2H - 96\lambda_2\lambda_3 + \frac{16}{3}\lambda_2H^3 + 48\lambda_2\delta_1\delta_{1,1} - \\ & - 96\lambda_2\delta_{1,1,1}
		- \frac{241}{36}H^4\delta_1 - \frac{1111}{36}H^3\delta_1^2 - \frac{63}{4}H^2\delta_1^3 + \frac{367}{4}H\delta_1^4 + 126\delta_1^5
		\end{split}
		\end{equation*}
		
		\item[] 
		\begin{equation*}
		\begin{split}
		m(3)&=\frac{419}{648}\lambda_1^5H - 9\lambda_1^5\delta_1 + \frac{17903}{972}\lambda_1^4H^2 - \frac{19763}{648}\lambda_1^4H\delta_1 -
		\frac{285}{4}\lambda_1^4\delta_1^2 + \\ & + \frac{57344}{27}\lambda_1^4\delta_{1,1} - \frac{401}{162}\lambda_1^3\lambda_2H + 15\lambda_1^3\lambda_3 +
		\frac{100795}{1944}\lambda_1^3H^3 - \\ & - \frac{16057}{972}\lambda_1^3H^2\delta_1 - \frac{100555}{648}\lambda_1^3H\delta_1^2 -
		\frac{861}{4}\lambda_1^3\delta_1^3 + \frac{614635}{54}\lambda_1^3\delta_1\delta_{1,1}  + \\ & + 15\lambda_1^3\delta_{1,1,1} - \frac{6433}{81}\lambda_1^2\lambda_2H^2 -
		32\lambda_1^2\lambda_2\delta_{1,1} + \frac{11561}{216}\lambda_1^2H^4 + \\ & + \frac{12349}{324}\lambda_1^2H^3\delta_1 +
		\frac{559}{12}\lambda_1^2H^2\delta_1^2  - \frac{120883}{324}\lambda_1^2H\delta_1^3 - \frac{1263}{4}\lambda_1^2\delta_1^4 + \\ & +
		\frac{198799}{9}\lambda_1^2\delta_1^2\delta_{1,1} - 22\lambda_1\lambda_2^2H - 52\lambda_1\lambda_2\lambda_3 - 151\lambda_1\lambda_2H^3  + \\ & +
		54\lambda_1\lambda_2\delta_1\delta_{1,1} - 52\lambda_1\lambda_2\delta_{1,1,1} + \frac{2845}{162}\lambda_1H^5 + \frac{19415}{324}\lambda_1H^4\delta_1 + \\ & +
		\frac{59303}{324}\lambda_1H^3\delta_1^2 + \frac{10946}{243}\lambda_1H^2\delta_1^3 - \frac{66367}{162}\lambda_1H\delta_1^4 -
		\frac{903}{4}\lambda_1\delta_1^5 + \\ & + 18519\lambda_1\delta_1^3\delta_{1,1} - 22\lambda_2^2H^2 - \frac{1333}{18}\lambda_2H^4  +
		86\lambda_2\delta_1^2\delta_{1,1} + 112\lambda_3^2 + \\ & + 112\lambda_3\delta_{1,1,1} - \frac{407}{216}H^6 + \frac{14521}{648}H^5\delta_1 +
		\frac{40205}{648}H^4\delta_1^2 + \frac{147917}{972}H^3\delta_1^3 - \\ & - \frac{8563}{243}H^2\delta_1^4 -
		\frac{104075}{648}H\delta_1^5 - 63\delta_1^6 + \frac{11377}{2}\delta_1^4\delta_{1,1}
		\end{split}
		\end{equation*}
		
		\item[] 
		\begin{equation*}
			\begin{split}
				h(1)&=\frac{3}{4}\lambda_1^3H + \frac{13}{4}\lambda_1^2H^2 + \frac{9}{4}\lambda_1^2H\delta_1 + \frac{13}{4}\lambda_1H^3 +  \frac{13}{2}\lambda_1H^2\delta_1 +
				\frac{9}{4}\lambda_1H\delta_1^2 + \\ & + \frac{3}{4}H^4 + \frac{13}{4}H^3\delta_1+ \frac{13}{4}H^2\delta_1^2 + \frac{3}{4}H\delta_1^3
			\end{split}
		\end{equation*}
	
		\item[] 
		\begin{equation*}
		\begin{split}
		r&=-\frac{7}{81}\lambda_1^8H + \frac{247145}{2916}\lambda_1^7H^2 - \frac{39125}{81}\lambda_1^7H\delta_1 - 20800\lambda_1^7\delta_{1,1} - \\ & -
		\frac{1286}{243}\lambda_1^6\lambda_2H + \frac{1573727}{8748}\lambda_1^6H^3  -\frac{400579}{162}\lambda_1^6H^2\delta_1 -
		\frac{618187}{162}\lambda_1^6H\delta_1^2 + \\ & + \frac{31710800}{81}\lambda_1^6\delta_1\delta_{1,1} - \frac{736943}{729}\lambda_1^5\lambda_2H^2  -
		24288\lambda_1^5\lambda_2\delta_{1,1} - \frac{2193853}{8748}\lambda_1^5H^4 - \\ & - \frac{4516739}{729}\lambda_1^5H^3\delta_1 -
		\frac{35110427}{2916}\lambda_1^5H^2\delta_1^2  - \frac{3448919}{243}\lambda_1^5H\delta_1^3 + \\ & +
		\frac{253855904}{81}\lambda_1^5\delta_1^2\delta_{1,1} - \frac{136}{3}\lambda_1^4\lambda_2^2H - \frac{2054561}{729}\lambda_1^4\lambda_2H^3+ \\ & +
		133280\lambda_1^4\lambda_2\delta_1\delta_{1,1} - \frac{2986483}{2916}\lambda_1^4H^5 - \frac{40469125}{4374}\lambda_1^4H^4\delta_1 - \\ & -
		\frac{52534855}{2916}\lambda_1^4H^3\delta_1^2- \frac{33507299}{1458}\lambda_1^4H^2\delta_1^3 -
		\frac{17079953}{486}\lambda_1^4H\delta_1^4 + \\ & + \frac{246525952}{27}\lambda_1^4\delta_1^3\delta_{1,1} + \frac{26584}{9}\lambda_1^3\lambda_2^2H^2  
		- 9248\lambda_1^3\lambda_2^2\delta_{1,1} - \frac{364370}{243}\lambda_1^3\lambda_2H^4 + \\ & + 911024\lambda_1^3\lambda_2\delta_1^2\delta_{1,1} -
		1152\lambda_1^3\lambda_3^2 - 1152\lambda_1^3\lambda_3\delta_{1,1,1}  - \frac{315269}{324}\lambda_1^3H^6 - \\ & -
		\frac{6063527}{729}\lambda_1^3H^5\delta_1 - \frac{80352425}{4374}\lambda_1^3H^4\delta_1^2 -
		\frac{8710724}{2187}\lambda_1^3H^3\delta_1^3  - \\ & - \frac{26273782}{729}\lambda_1^3H^2\delta_1^4 -
		\frac{13515761}{243}\lambda_1^3H\delta_1^5 + \frac{41134658}{3}\lambda_1^3\delta_1^4\delta_{1,1} + 896\lambda_1^2\lambda_2^3H + \\ & +
		256\lambda_1^2\lambda_2^2\lambda_3  + \frac{225704}{27}\lambda_1^2\lambda_2^2H^3 - 11680\lambda_1^2\lambda_2^2\delta_1\delta_{1,1} +
		256\lambda_1^2\lambda_2^2\delta_{1,1,1} + \\ & + \frac{19484}{9}\lambda_1^2\lambda_2H^5 + 1716232\lambda_1^2\lambda_2\delta_1^3\delta_{1,1} -
		\frac{70385}{324}\lambda_1^2H^7 - \frac{716609}{162}\lambda_1^2H^6\delta_1 - \\ & - \frac{3175405}{243}\lambda_1^2H^5\delta_1^2 +
		\frac{14917697}{2187}\lambda_1^2H^4\delta_1^3 + \frac{27927485}{1458}\lambda_1^2H^3\delta_1^4 - \\ & -
		\frac{24157867}{486}\lambda_1^2H^2\delta_1^5 - \frac{12315655}{243}\lambda_1^2H\delta_1^6 +
		11373175\lambda_1^2\delta_1^5\delta_{1,1}  + \\ & + 1664\lambda_1\lambda_2^3H^2 - 1152\lambda_1\lambda_2^3\delta_{1,1} +
		192920/27\lambda_1\lambda_2^2H^4 + 7328\lambda_1\lambda_2^2\delta_1^2\delta_{1,1} + \\ & + 5824\lambda_1\lambda_2\lambda_3^2+
		5824\lambda_1\lambda_2\lambda_3\delta_{1,1,1} + \frac{66191}{27}\lambda_1\lambda_2H^6 + 1353816\lambda_1\lambda_2\delta_1^4\delta_{1,1} + \\ & +
		\frac{12985}{108}\lambda_1H^8 - \frac{104593}{81}\lambda_1H^7\delta_1  - \frac{1752295}{324}\lambda_1H^6\delta_1^2 +
		\frac{906349}{729}\lambda_1H^5\delta_1^3 + \\ & + \frac{57919459}{2187}\lambda_1H^4\delta_1^4 + \frac{6920350}{729}\lambda_1H^3\delta_1^5
		- \frac{53724649}{1458}\lambda_1H^2\delta_1^6 - \\ & - \frac{5743493}{243}\lambda_1H\delta_1^7 + 4958470\lambda_1\delta_1^6\delta_{1,1} -
		1152\lambda_2^3\lambda_3 + 768\lambda_2^3H^3 - \\ & - 1152\lambda_2^3\delta_1\delta_{1,1} - 1152\lambda_2^3\delta_{1,1,1} +
		\frac{16064}{9}\lambda_2^2H^5 + 9760\lambda_2^2\delta_1^3\delta_{1,1} + \frac{5399}{9}\lambda_2H^7 + \\ & + 391040\lambda_2\delta_1^5\delta_{1,1}  -
		10976\lambda_3^3 - 10976\lambda_3^2\delta_{1,1,1} + \frac{171}{4}H^9 - \frac{7903}{54}H^8\delta_1 - \\ & -
		\frac{304223}{324}H^7\delta_1^2 - \frac{365225}{486}H^6\delta_1^3  + \frac{4136302}{729}H^5\delta_1^4 +
		\frac{43734445}{4374}H^4\delta_1^5 - \\ & - \frac{3256102}{2187}H^3\delta_1^6 - \frac{14121601}{1458}H^2\delta_1^7 -
		\frac{1042615}{243}H\delta_1^8 + 887989\delta_1^7\delta_{1,1}
		\end{split}
		\end{equation*}

	\end{itemize}
\end{remark}

\bibliographystyle{plain}
\bibliography{Bibliografia}

\begin{bibdiv}
\begin{biblist}

\bib{AbOlVis}{article}{
      author={Abramovich, D.},
      author={Olsson, M.},
      author={Vistoli, A.},
       title={{Twisted stable maps to tame Artin stacks}},
        date={2011},
     journal={Journal of Algebraic Geometry},
      volume={20},
       pages={399\ndash 477},
}

\bib{ArVis}{article}{
      author={Arsie, A.},
      author={Vistoli, A.},
       title={{Stacks of cyclic covers of projective spaces}},
        date={2004},
     journal={Compos. Math. 140, no.3, 647-666},
}

\bib{Cat}{article}{
      author={Catanese, F.},
       title={{Pluricanonical Gorenstein Curves}},
        date={1982},
     journal={Enumerative Geometry and Classical Algebraic Geometry},
       pages={51\ndash 95},
         url={https://doi.org/10.1007/978-1-4684-6726-0_4},
}

\bib{CanLar}{article}{
      author={Canning, S.},
      author={Larson, H.},
       title={{The Chow rings of the moduli spaces of curves of genus 7, 8, and
  9}},
        date={2021},
         url={https://arxiv.org/abs/2104.05820},
}

\bib{DiLor}{article}{
      author={Di~Lorenzo, A.},
       title={{The Chow ring of the stack of hyperelliptic curves of odd
  genus}},
        date={2019},
     journal={Int. Math. Res. Not. IMRN rnz 101,
  https://doi.org/10.1093/irmn/rnz 101},
}

\bib{DiLor2}{article}{
      author={Di~Lorenzo, A.},
       title={{Picard group of moduli of curves of low genus in positive
  characteristic}},
        date={2020jun},
     journal={Manuscripta Mathematica},
}

\bib{DiLorFulVis}{article}{
      author={Di~Lorenzo, A.},
      author={Fulghesu, D.},
      author={Vistoli, A.},
       title={{The integral Chow ring of the stack of smooth non-hyperelliptic
  curves of genus three}},
        date={2021},
     journal={Transactions of the American Mathematical Society},
      volume={374},
      number={08},
       pages={5583\ndash 5622},
}

\bib{DiLorPerVis}{misc}{
      author={Di~Lorenzo, A.},
      author={Pernice, M.},
      author={Vistoli, A.},
       title={{Stable cuspidal curves and the integral Chow ring of
  $\overline{\mathscr{M}}_{2,1}$}},
   publisher={arXiv},
        date={2021},
         url={https://arxiv.org/abs/2108.03680},
}

\bib{DiLorVis}{article}{
      author={Di~Lorenzo, A.},
      author={Vistoli, A.},
       title={{Polarized twisted conics and moduli of stable curves of genus
  two}},
        date={2021},
         url={https://arxiv.org/abs/2103.13204},
}

\bib{EdFul2}{article}{
      author={Edidin, D.},
      author={Fulghesu, D.},
       title={{The integral Chow ring of the stack of at most 1-nodal rational
  curves}},
        date={2006},
}

\bib{EdGra}{article}{
      author={Edidin, D.},
      author={Graham, W.},
       title={{ Equivariant intersection theory}},
        date={1998},
     journal={Inv. math. 131, 595-634},
}

\bib{Est}{article}{
      author={Esteves, E.},
       title={{The stable hyperelliptic locus in genus 3: an application of
  Porteous Formula}},
        date={2013},
     journal={Journal of Pure and Applied Algebra},
      volume={220},
       pages={845\ndash 856},
}

\bib{Fab}{article}{
      author={Faber, C.},
       title={{Chow Rings of Moduli Spaces of Curves I: The Chow Ring of
  $\overline{\mathcal M}_3$}},
        date={1990},
     journal={Annals of Mathematics},
      volume={132},
       pages={331\ndash 419},
         url={http://www.jstor.org/stable/1971525},
}

\bib{FulVis}{article}{
      author={Fulghesu, D.},
      author={Vistoli, A.},
       title={{The Chow Ring of the Stack of Smooth Plane Cubics}},
        date={2016},
     journal={The Michigan Mathematical Journal},
      volume={67},
}

\bib{Iza}{article}{
      author={Izadi, E.},
       title={{The Chow Ring of the Moduli Space of Curves of Genus 5}},
        date={1995},
       pages={267\ndash 303},
}

\bib{Lar}{article}{
      author={Larson, E.},
       title={{The integral Chow ring of $\overline{\mathscr{M}}_2$}},
        date={2019},
     journal={arXiv:1904.08081},
}

\bib{MolVis}{article}{
      author={Molina~Rojas, L.~A.},
      author={Vistoli, A},
       title={{On the Chow rings of classifying spaces for classical groups}},
        date={2006},
     journal={Rendiconti del Seminario Matematico della Università di Padova},
      volume={116},
       pages={271\ndash 298},
         url={http://eudml.org/doc/108697},
}

\bib{Mum}{article}{
      author={Mumford, D.},
       title={{ Towards an Enumerative Geometry of the Moduli Space of Curves
  }},
        date={1983},
     journal={Artin M., Tate J. (eds) Arithmetic and Geometry. Progress in
  Mathematics, vol 36. Birkh\"auser, Boston, MA},
}


\bib{Per1}{article}{
      author={Pernice, M.},
       title={The moduli stack of $A_r$-stable curves},
        date={2023},
     journal={arXiv},
         url={https://arxiv.org/abs/2302.10877},
}
\bib{Per2}{article}{
      author={Pernice, M.},
       title={Hyperelliptic $A_r$-stable curves (and their moduli stack)},
        date={2023},
     journal={arXiv},
         url={https://arxiv.org/abs/2302.11456},
}



\bib{PenVak}{article}{
      author={Penev, N.},
      author={Vakil, R.},
       title={{The Chow ring of the moduli space of curves of genus 6}},
        date={2013},
     journal={Algebraic Geometry},
      volume={2},
}

\bib{Rom}{article}{
      author={Romagny, M.},
       title={{Group actions on stacks and applications}},
        date={2005},
     journal={Michigan Math. J., Volume 53, Issue 1, 209-236},
}

\bib{Tot}{book}{
      author={Totaro, B.},
       title={{Group Cohomology and Algebraic Cycles}},
      series={Cambridge Tracts in Mathematics},
   publisher={Cambridge University Press},
        date={2014},
}

\bib{Vis1}{article}{
      author={Vistoli, A.},
       title={{Intersection theory of Algebraic Stacks and on their moduli
  spaces}},
        date={1989},
     journal={Inventationes Mathematicae, Springer-Verlag},
}

\end{biblist}
\end{bibdiv}

\end{document}